\documentclass[]{article}

\usepackage{lipsum}
\usepackage{amsfonts}
\usepackage{graphicx}
\usepackage{epstopdf}
\usepackage{algorithmic}
\usepackage{hyperref}
\usepackage[margin=1in]{geometry}
\usepackage{amsmath}
\usepackage{cancel}
\usepackage{algorithm}
\usepackage{amsthm}

\usepackage{amsopn}

\usepackage{multicol}

\newcommand{\mybudget}{c_{\rm{tot}}}

\newcommand{\mymf}{\rm{mf}}
\newcommand{\mymc}{\rm{mc}}

\DeclareMathOperator*{\argmin}{arg\,min}

\newcommand{\email}[1]{\protect\href{mailto:#1}{#1}}

\newtheorem{theorem}{Theorem}
\newtheorem{lemma}{Lemma}
\newtheorem{proposition}{Proposition}
\newtheorem{remark}{Remark}

\title{Multifidelity Proper Orthogonal Decomposition}

\author{Nicole Aretz\thanks{Oden Institute for Computational Engineering and Sciences, University of Texas at Austin, Austin, TX 
  (\email{nicole.aretz@austin.utexas.edu}, \email{kwillcox@oden.utexas.edu}
  ).}
\and Karen Willcox\footnotemark[2] 
}

\begin{document}

\maketitle

\begin{abstract}
This paper introduces a multifidelity formulation that reduces the computational cost of the proper orthogonal decomposition (POD) of a high-fidelity model by leveraging data from cheaper, lower-fidelity models.
POD is a prevalent technique for extracting a low-dimensional basis from training data to achieve subsequent dimension reduction or reduced-order modeling.
In scientific and engineering applications, the training data are typically numerical snapshot solutions of a high-fidelity model, and computation of a sufficiently rich snapshot set can be prohibitively expensive, especially when sampling over a high-dimensional parameter space.
Insufficient snapshot training data risks overfitting and poor generalizability of the POD basis to outside the training regime.
Our multifidelity POD (MFPOD) formulation reallocates computational budget to cheaper, low-fidelity models that can be sampled more extensively. 
MFPOD then weights high- and low-fidelity snapshot data via a control-variate formulation to guarantee an unbiased estimate of the expected high-fidelity least-squares projection error.
The MFPOD subspace is chosen to minimize the estimate of this projection error, and converges in probability to the same subspace as single-fidelity POD in the limit of an arbitrarily large budget.
For restrictive computational budgets, the MFPOD cost function has (under some assumptions) lower variance than the POD cost function, which makes the MFPOD subspace more robust against variations in the training data and thus less prone to overfitting.
For a numerical example modeling the velocity of the Pine Island glacier, 
MFPOD achieves the same accuracy as single-fidelity POD with
an order of magnitude reduction in the offline computational cost of snapshot generation.
\end{abstract}

\section{Introduction}
\label{sec:introduction}

We present a multifidelity approach to computing the proper orthogonal decomposition (POD, \cite{lumley1967structure, sirovich1987turbulence}) for a high-fidelity model at reduced computational cost.
POD --- also known as principal component analysis, Karhunen-Lo\'{e}ve expansion, and empirical orthogonal eigenfunctions --- is a powerful data-driven dimensionality reduction technique to identify a hierarchy of POD modes that reconstruct training data optimally in a least squares sense.
The training data are typically numerical solutions (snapshots) of a computational model of the problem at hand and are expensive to compute.
The computational cost of generating training data is often the primary barrier for applying dimension reduction and reduced-order modeling in real-world large-scale scientific applications.
When constrained compute resources only permit a small number of snapshots, the POD space may miss important features.
While low-fidelity models --- such as coarser discretizations or simplified physics --- are cheaper to evaluate and thus permit more snapshots for the same computational cost, a POD space built on their data is restricted by their low accuracy.
In contrast, our multifidelity POD (MFPOD) formulation leverages the correlation between high- and low-fidelity models 
to overcome
the accuracy and snapshot sample size constraints of single-fidelity POD approaches.

Low-fidelity models may include coarser discretizations of a given set of governing equations or approximate models derived by simplifying physical assumptions.
Compared to single-fidelity approaches, multifidelity methods combine data from multiple models to improve estimates of quantities of interest in terms of compute time or accuracy.
Multifidelity methods are widely used in forward \cite{chaudhuri2021mfegra, Croci2023, dodwell2019multilevel, giles2008multilevel, giles2015multilevel, gorodetsky2020generalized, heinrich2001multilevel, Ng2014, peherstorfer2016multifidelity, peherstorfer2016optimal, romor2023multi, Schaden2020} and inverse \cite{cui2024multilevel, cui2016scalable, peherstorfer2019transport, tezzele2023multifidelity} uncertainty quantification to alleviate the cost of sampling.
In the context of constructing reduced-order or surrogate models, low-fidelity models have been used to approximate covariance matrices in polynomial chaos expansions \cite{doostan2007stochastic, mohammadi2019stochastic, raisee2015non}, 
select interpolation points in multifidelity stochastic collocation \cite{chen2014comparison, narayan2014stochastic, zhu2017multi, zhu2014computational},
obtain basis coefficient vectors in bifidelity interpolative decomposition \cite{cutforth2026bi, doostan2016bi, hampton2018practical},
accelerate greedy sampling \cite{ chellappa2020adaptive2,chellappa2022adaptive, feng2024posteriori, feng2023multi}, and improve training in multifidelity regression \cite{conti2024multi,  sella2025projection, shi2020multi}.
While these latter approaches use multifidelity data for model reduction, we focus here on the construction of the POD subspace. In contrast to previous work, we are specifically targeting the upfront computational cost of identifying the POD subspace.
This space can then be paired with any method of choice for subsequent tasks, such as dimension reduction, model reduction, or the multifidelity methods referenced above.

Our MFPOD approach combines the least squares optimization formulation that defines the POD basis with the multifidelity Monte Carlo (MFMC) method \cite{Ng2014, peherstorfer2016optimal}.
Based on control variates, MFMC exploits low-fidelity surrogate models to accelerate the estimation of a statistical quantity of interest (such as the expectation, variance, or sensitivity indices \cite{qian2018multifidelity}) compared to Monte Carlo sampling.
MFMC is compatible with restrictive computational budgets \cite{gruber2023multifidelity}, and has been applied successfully to large-scale uncertainty quantification problems \cite{aretz2025multifidelity, gruber2023multifidelity}.
Related to our MFPOD approach are multifidelity covariance matrix estimators \cite{lam2020multifidelity,pmlr-v202-maurais23a, maurais2025multifidelity}.
Here, \cite{lam2020multifidelity} on multifidelity active subspaces is the most similar approach to ours as it approximates a dominant eigenspace;
however, our approach is distinct through its estimator structure and the active tuning of model weights as the reduced space expands.

The contribution of our work is MFPOD --- a new method that builds a single, high-fidelity POD space from multifidelity data, thus reducing the cost of generating training data for POD. Theoretical analysis of the MFPOD method provides a convergence analysis and numerical results illustrate the realized numerical gains.
Past work has considered related problems but has not tackled this specific challenge.
The multilevel POD method in \cite{behzad2018multilevel} restructures the single-fidelity POD method to achieve a computational speed-up when the number of snapshots is large.
Further, \cite{grassle2018pod, grassle2019pod, ullmann2016pod} perform POD on a variety of discretizations but these are chosen adaptively to guarantee the same level of fidelity throughout all training data rather than explicitly targeting cheap low-fidelity snapshots as in our new MFPOD method.
Greedy sampling approaches, while being rate-optimal (compared to the Kolmogorov $n$-width) for elliptic partial differential equations  \cite{binev2011convergence, dahmen2015best, devore2013greedy} in the limit of arbitrarily many snapshots,
lead to a reduced space dimension that remains limited in the low-budget regime by the number of available snapshots.

The remainder of this paper is organized as follows:
Section~\ref{sec:mfpod} introduces the MFPOD cost function and characterizes it through an eigenvalue problem.
Section~\ref{sec:discrete} provides a discrete MFPOD formulation and an algorithmic description for computing the MFPOD eigenvalues and modes in practice.
We then analyze the multifidelity eigenvalue problem further in Section~\ref{sec:analysis} before demonstrating the MFPOD approach on two numerical examples in Section~\ref{sec:results}.
Finally, Section~\ref{sec:conclusion} concludes the paper.

\section{Multifidelity Proper Orthogonal Decomposition}
\label{sec:mfpod}

This section summarizes key properties of the POD and presents the formulation of the proposed MFPOD method.

\subsection{Proper Orthogonal Decomposition}
For any given reduced dimension $r \in \mathbb{N}$, POD identifies a reduced space $V$, $\dim V = r$, such that the projection $\Pi_V$ onto $V$ optimally reconstructs training data in a least-squares sense.
Here, we specifically consider training \textit{snapshots} $u_0(\theta_1), \dots, u_0(\theta_m)$ in a separable Hilbert space $H$ that are obtained from a parameterized high-fidelity model for $m \in \mathbb{N}$ training parameters $\theta_1, \dots, \theta_m \in \mathbb{R}^d$.
The POD subspace $V$ is then characterized through
\begin{align}\label{eq:POD:min}
    V = \argmin_{V \subset H,\, \dim V = r} \, \frac{1}{m} \sum_{i=1}^m \|u_0(\theta_i)-\Pi_V u_0(\theta_i)\|_H^2.
\end{align}
An appropriate choice for the reduced dimension $r$ and orthonormal basis functions $v_1, \dots, v_r \in H$ of $V$ can be obtained through the method of snapshots:
First, compute the gramian matrix $\mathbf{C} \in \mathbb{R}^{m \times m}$ with entries $[\mathbf{C}]_{i,j} := (u_0(\theta_i), u_0(\theta_j))_H$;
second, identify its eigenvalue-eigenvector pairs $(\lambda_i, \mathbf{w}_i) \in \mathbb{R}_{\ge 0} \times \mathbb{R}^m$, $\mathbf{C} \mathbf{w}_i = \lambda_i \mathbf{w}_i$ such that $\lambda_1 \ge \dots \ge \lambda_m \ge 0$ and $\mathbf{w}_1, \dots, \mathbf{w}_m$ are orthonormal in $\mathbb{R}^m$;
third, using that
\begin{align}\label{eq:POD:min:eigvals}
    \min_{V\subset H, \dim V = r} \, \frac{1}{m} \sum_{i=1}^m \|u_0(\theta_i)-\Pi_V u_0(\theta_i)\|_H^2
    = \sum_{j=r+1}^m \lambda_j,
\end{align}
identify $r \le m$ to meet any target accuracy;
fourth, assemble the basis functions $v_j := \frac{1}{\sqrt{m \lambda_j}} \sum_{i=1}^m [\mathbf{w}_j]_i \, u_0(\theta_i)$ for $j=1,2,\ldots,r$ and set $V := \text{span} \{v_1, \dots, v_r \} \subset H$.

\subsection{POD cost function}
We now consider the training parameters $\theta_1, \dots, \theta_m$ to be independent and identically distributed (i.i.d.)~samples $\theta_i : \Omega \rightarrow \mathbb{R}^d$ defined on a probability space $(\Omega, \mathcal{F}, \mathbb{P})$, each with distribution $\mu$.
In this setting, the cost function 
\begin{align}
    \mathcal{J}_{\mymc}(V) := \frac{1}{m} \sum_{i=1}^m \|u_0(\theta_i)-\Pi_V u_0(\theta_i)\|_H^2
\end{align}
from the POD minimization problem \eqref{eq:POD:min} is the Monte Carlo approximation of the expected projection error
\begin{align}
    \mathcal{J}(V) := \mathbb{E}_{\theta \sim \mu} \left[ \|u_0(\theta)-\Pi_V u_0(\theta)\|_H^2 \right].
\end{align}
As such, its mean squared error (MSE) is
\begin{align}\label{eq:POD:mse}
    \text{MSE}[\mathcal{J}_{\mymc}(V)] 
    &\hspace{-0.2em}:=\hspace{-0.2em} \mathbb{E}_{\theta_1, \dots, \theta_m \sim \mu}\hspace{-0.25em} \left[(\mathcal{J}_{\mymc}(V)\hspace{-0.2em} - \hspace{-0.2em}\mathcal{J}(V))^2\right] \hspace{-0.25em}
    = \hspace{-0.25em} \frac{1}{m} \mathbb{V}_{\hspace{-0.1em}\theta \sim \mu}(\|u_0(\theta)\hspace{-0.2em}-\hspace{-0.2em}\Pi_V u_0(\theta)\|_H^2\hspace{-0.2em})
\end{align}
where $\mathbb{V}_{\theta \sim \mu}(\|u_0(\theta)-\Pi_V u_0(\theta)\|_H^2)$ denotes the variance of the map $\theta \mapsto \|u_0(\theta)-\Pi_V u_0(\theta)\|_H^2$ with respect to the distribution $\mu$.
Note that, in order for the variance to be finite for all subspaces $V \subset H$, we need $\|u_0(\theta)\|_H$ to be fourth-order integrable over the parameter domain $\text{supp}(\mu)$ w.r.t.\ $\mu$, i.e.
$\mathbb{E}_{\theta \sim \mu}[\|u_0(\theta)\|_H^4] < \infty$.
The expression \eqref{eq:POD:mse} reveals a fundamental challenge for the POD: 
The $1/m$ decay of the MSE w.r.t.\ the number of samples $m$.
Generally, choosing the sample size $m$ too small in the POD optimization \eqref{eq:POD:min} for $V$ bears the risks of overfitting and a high generalization error outside the training regime.
However, in practice the sample size $m$ cannot be chosen arbitrarily large due to the cost (e.g., CPUh) of generating the snapshots:
If $c_0 > 0$ is the cost of obtaining $u_0(\theta)$ for any one $\theta \in \mathbb{R}^d$ and there is a finite sampling budget $\mybudget < \infty$, then the sample size $m$ is naturally restricted by $\lfloor \mybudget / c_0 \rfloor$.

\subsection{Multifidelity POD cost function}

The core idea of our MFPOD method is that using the MFMC method \cite{Ng2014, peherstorfer2016optimal} we can define a multifidelity cost function $\mathcal{J}_{\mymf}$ that is statistically unbiased 
\begin{align}\label{eq:MFPOD:unbiased}
    \mathbb{E}[\mathcal{J}_{\mymf}(V)] &= \mathbb{E}[\mathcal{J}(V)] = \mathbb{E}_{\theta \sim \mu}[\|u_0(\theta)-\Pi_V u_0(\theta)\|_H^2]
\end{align}
for all subspaces $V \subset H$,
and has --- under some assumptions --- a smaller MSE than the Monte Carlo cost function $\mathcal{J}_{\mymc}$,
\begin{align}
    \text{MSE}[\mathcal{J}_{\mymf}(V)] 
    &= \mathbb{E}[(\mathcal{J}_{\mymf}(V) - \mathcal{J}(V))^2] \le \text{MSE}[\mathcal{J}_{\mymc}(V)].
\end{align}
With $\mathcal{J}_{\mymf}$ we then define an MFPOD minimization problem
\begin{align}\label{eq:MFPOD:min}
    V = \argmin_{V \subset H,\, \dim V=r} \mathcal{J}_{\mymf}(V)
\end{align}
to obtain our MFPOD subspace.
As long as $\text{MSE}[\mathcal{J}_{\mymf}(V)] \le \text{MSE}[\mathcal{J}_{\mymc}(V)]$, the basis $V$ obtained through \eqref{eq:MFPOD:min} is less prone to overfitting than the standard POD.


To define the multifidelity cost function $\mathcal{J}_{\mymf}$, we leverage snapshots from $L \ge 1$ low-fidelity models that are both cheaper to evaluate but also of lower accuracy than the high-fidelity model itself.
These low-fidelity models might, for example, stem from coarser numerical discretizations or simplified physics;
formally we only require that they map input parameters $\theta \in \text{supp} (\mu) \subset \mathbb{R}^d$ to states $u_{\ell}(\theta)$ in the Hilbert space $H$ with the integer index $\ell$, $1 \le \ell \le L$, indicating the model number.
We assume without loss of generality that models are numbered by their costs $c_{\ell}$ such that $c_0 > c_1 \ge \dots \ge c_L > 0$.
For sample sizes $m_0 < m_1 < \dots < m_L$ and weights $\alpha_1, \dots, \alpha_L \in \mathbb{R}$, we then define the multifidelity cost function
\begin{equation} \label{eq:Jmf:indefinite}
\begin{aligned}
    \mathcal{J}_{\mymf}(V) &:= \frac{1}{m_0} \sum_{i=1}^{m_0} \|u_0(\theta_i)-\Pi_V u_0(\theta_i)\|_H^2 \\
    &+ \sum_{\ell=1}^L \left(
    \frac{\alpha_{\ell}}{m_{\ell}} \sum_{i=1}^{m_{\ell}} \|u_{\ell}(\theta_i)-\Pi_V u_{\ell}(\theta_i)\|_H^2
    - \frac{\alpha_{\ell}}{m_{\ell-1}} \sum_{i=1}^{m_{\ell-1}} \|u_{\ell}(\theta_i)-\Pi_V u_{\ell}(\theta_i)\|_H^2
    \right)
\end{aligned}
\end{equation}
with i.i.d.\ samples $\theta_1, \dots, \theta_{m_L} \sim \mu$.
Note that, as in MFMC, samples are shared between the different sums; the total sampling cost for $\mathcal{J}_{\mymf}$ is hence $\sum_{\ell=0}^L m_{\ell} c_{\ell}$.
Applying the expectation over $\theta_1, \dots, \theta_{m_L} \sim \mu$ to $\mathcal{J}_{\mymf}(V)$ immediately yields \eqref{eq:MFPOD:unbiased}.

Following \cite{peherstorfer2016optimal}, Lemma 3.3, the MSE of $\mathcal{J}_{\mymf}(V)$ is
\begin{align}\label{eq:MFPOD:MSE}
    \text{MSE}(\mathcal{J}_{\mymf}(V))
    &= \frac{\sigma_0^2(V)}{m_0} + \sum_{\ell=1}^L \bigg(\frac{1}{m_{\ell-1}}-\frac{1}{m_{\ell}} \bigg) \bigg(\alpha_{\ell}^2 \sigma_{\ell}^2(V) - 2 \alpha_{\ell} \sigma_{0, \ell}(V) \bigg)
\end{align}
with the model variances and covariances
\begin{align*}
    \sigma_{\ell}^2(V) &:= \mathbb{V}_{\theta \sim \mu} \left(  \|u_{\ell}(\theta)-\Pi_V u_{\ell}(\theta)\|_H^2 \right), \\
    \sigma_{0, \ell}(V) &:= \text{Cov}_{\theta \sim \mu} \left(\|u_{0}(\theta)-\Pi_V u_{0}(\theta)\|_H^2, \|u_{\ell}(\theta)-\Pi_V u_{\ell}(\theta)\|_H^2 \right),
\end{align*}
defined for $\ell=0,1,\ldots,L$.
In \eqref{eq:MFPOD:MSE}, $\frac{1}{m_{\ell-1}}-\frac{1}{m_{\ell}} > 0$ for $\ell = 1, \dots, L$ because the sample sizes $m_0 < m_1 < \dots < m_L$ are increasing.
Therefore, for $\ell = 1, \dots, L$, the $\ell$-th low-fidelity model 
contributes towards reducing the MSE of $\mathcal{J}_{\mymf}(V)$ iff $\alpha_{\ell}^2 \sigma_{\ell}^2(V) - 2 \alpha_{\ell} \sigma_{0, \ell}(V) < 0$.
The optimal choice for $\alpha_{\ell}$,  $\ell=1,2,\ldots,L$ is hence
\begin{align}\label{eq:MFPOD:alphastar}
    \alpha_{\ell}^*(V) := \left\{
    \begin{array}{ll}
        \sigma_{0, \ell}(V) / \sigma_{\ell}^2(V) & \text{if } \sigma_{\ell}^2(V)>0, \\
        0 & \text{if } \sigma_{\ell}^2(V) = 0.
    \end{array}
    \right.
\end{align}
With the choice $\alpha_{\ell}=\alpha_{\ell}^*(V)$, $\ell = 1, \dots, L$, the MSE of $\mathcal{J}(V)$ is minimal and given by
\begin{align}\label{eq:MFPOD:mse:min}
    \text{MSE}(\mathcal{J}_{\mymf}(V))
    &= \frac{\sigma_0^2(V)}{m_0} - \hspace{-0.5em}\sum_{\substack{\ell=1 \\ \sigma^2_{\ell}(V) \neq 0}}^L \hspace{-0.5em}\bigg(\frac{1}{m_{\ell-1}}-\frac{1}{m_{\ell}} \bigg)
    \frac{\sigma_{0, \ell}^2(V)}{\sigma_{\ell}^2(V)}
    \, \le \, \frac{\sigma_0^2(V)}{m_0}.
\end{align}
In the worst case, $\text{MSE}(\mathcal{J}_{\mymf}(V))$ in \eqref{eq:MFPOD:mse:min} equals the MSE of Monte Carlo sampling with $m_0$ high-fidelity samples.
In a comparison at the same computational budget $\mybudget = \sum_{\ell = 0}^L c_{\ell}m_{\ell}$, the Monte Carlo approximation may use $m_{\mymc} = \lfloor \mybudget / c_0 \rfloor$ samples of the high-fidelity model, yielding the MSE $\sigma_0^2(V) / m_{\mymc}$.
With optimal control variate weights $\alpha_{\ell} = \alpha_{\ell}^*(V)$, $\ell=1,2,\ldots,L$, the multifidelity cost function $\mathcal{J}_{\mymf}(V)$ achieves a smaller MSE than Monte Carlo sampling for the same computational budget iff
\begin{align}
    1 - \sum_{\ell=1}^L (\frac{m_0}{m_{\ell-1}}-\frac{m_0}{m_{\ell}}) \frac{\sigma_{0, \ell}^2(V)}{\sigma_{\ell}^2(V)} < \frac{m_0}{m_{\mymc}}.
\end{align}
This formula can be used to evaluate if a low-fidelity model is beneficial for the MFPOD cost function $\mathcal{J}_{\mymf}(V)$, or if it would be more favorable to re-allocate its computational budget to the remaining models.

\subsection{Multifidelity POD optimization}\label{sec:MFPOD-opt}

The optimization problem \eqref{eq:MFPOD:min} of identifying a reduced space $V \subset H$, $\dim V = r$, that minimizes $\mathcal{J}_{\mymf}(V)$ can be solved through an eigenvalue problem.
To this end, define $\mathcal{C}_{\mymf} : H \rightarrow H$ through
\begin{equation}
\begin{aligned}\label{eq:Cmf}
    \mathcal{C}_{\mymf} v &:= \frac{1}{m_0} \sum_{i=1}^{m_0} (u_0(\theta_i), v)_H u_0(\theta_i) \\
    &\quad+\sum_{\ell=1}^L \left(
    \frac{\alpha_{\ell}}{m_{\ell}} \sum_{i=1}^{m_{\ell}} (u_{\ell}(\theta_i), v)_H u_{\ell}(\theta_i)
    - \frac{\alpha_{\ell}}{m_{\ell-1}} \sum_{i=1}^{m_{\ell-1}} (u_{\ell}(\theta_i), v)_H u_{\ell}(\theta_i)
    \right),
\end{aligned}
\end{equation}
and note that $\mathcal{C}_{\mymf}$ is linear, self-adjoint, and compact.
The operator $\mathcal{C}_{\mymf}$ connects to the MFPOD optimization \eqref{eq:MFPOD:min} through the following lemma which we prove in Appendix~\ref{sec:appendix}.

\begin{lemma}\label{thm:costfunction}
    Define $\mathcal{C}_{\mymf} : H \rightarrow H$ through \eqref{eq:Cmf} and let $(\lambda_j, v_j) \in \mathbb{R} \times H$, $1 \le j \le \dim H \in \mathbb{N} \cup \{\infty\}$, be its eigenvalue-eigenfunction pairs, ordered such that $\lambda_1 \ge \dots \ge \lambda_r \ge \sup_{j > r} \lambda_j$ and $\{v_j\}_{j=1}^{\dim H}$ orthonormal. 
    Then 
    \begin{align}\label{eq:MFPOD:series}
    \mathcal{J}_{\mymf}(V) &= \sum_{j=1}^{\dim H} \lambda_j (1-\| \Pi_V v_j\|_H^2).
    \end{align}
    In particular, $\mathcal{J}_{\mymf}(V)$ is minimized by $\text{span}(v_1, \dots, v_r)$, and this minimizer is unique iff $\lambda_r > \sup_{j > r} \lambda_j$.
\end{lemma}

Lemma \ref{thm:costfunction} characterizes the MFPOD minimization \eqref{eq:MFPOD:min} through the eigenvalue-eigenfunction pairs $(\lambda_j, v_j) \in \mathbb{R} \times H$ of $\mathcal{C}_{\mymf}$. 
In Section \ref{sec:analysis}, we show that, if the sampling ratios $m_{\ell} / m_0$ defined for $\ell = 0, \dots, L$ are fixed integers, then in the limit of an aribitrarily large sampling buget $\mybudget = \sum_{\ell=0}^L m_{\ell}c_{\ell} \rightarrow \infty$ the multifidelity operator $\mathcal{C}_{\mymf}$ converges in probability in the Hilbert-Schmidt norm to the second-moment operator
\begin{align}\label{eq:C}
    \mathcal{C} &: H \rightarrow H, &\mathcal{C}v &:= \mathbb{E}_{\theta \sim \mu}[(u_0(\theta), v)_H u_0(\theta)].
\end{align}
The operator $\mathcal{C}$ is linear, self-adjoint, and positive semi-definite by definition.
$\mathcal{C}$ is also trace-class if $\mathbb{E}[\|u_0(\theta)\|_H^2] < \infty$. 
Letting $(\lambda_j^*, v_j^*) \in \mathbb{R}_{\ge 0} \times H$, for integers $j$ with $1 \le j \le \dim H$, be eigenvalue-eigenfunction pairs of $\mathcal{C}$ such that $\{v_j^*\}_{j=1}^{\dim H}$ are orthonormal and $\lambda_1^* \ge \lambda_2^* \ge \dots \ge 0$, 
we show in Section \ref{sec:analysis} that
\begin{align}
\mathbb{E}\left[\bigg(\sum_{j=1}^r \lambda_j^* - \sum_{j=1}^r \lambda_j\bigg)^2 \right] & \in \mathcal{O}\bigg(\frac{r}{\mybudget}\bigg).
\label{eq:bound:sec2:eigvals}
\end{align}
That is, the mean squared error of $\sum_{j=1}^r \lambda_j$ approximating $\sum_{j=1}^r \lambda_j^*$ decays to zero at a (worst case) rate inversely proportional to the sampling budget $\mybudget = \sum_{\ell=0}^L m_{\ell}c_{\ell}$.
Moreover, if $\lambda_{r}^* > \lambda_{r+1}^*$ then the mean squared residual error 
\begin{align}\label{eq:bound:sec2:eigvecs}
\mathbb{E}\left[ \bigg(\sum_{j=1}^r \|v_j - \Pi_{V_*}v_j\|_H^2\bigg)^2
        \right] 
        &\in \mathcal{O}\bigg( \frac{r}{\mybudget(\lambda_{r}^*-\lambda_{r+1}^*)^2}\bigg), 
\end{align}
between the multifidelity eigenfunctions $v_1, \dots, v_r$ and their projections onto $V_* := \text{span}(v_1^*, \dots, v_r^*)$
also converges to zero for $\mybudget \rightarrow \infty$ at the (worst case) rate $\mathcal{O}(r/\mybudget)$, albeit with a constant scaling as $(\lambda_{r}^*-\lambda_{r+1}^*)^{-2}$ in the spectral gap $\lambda_{r}^*-\lambda_{r+1}^* > 0$.
Consequently, in the limit of an arbitrarily large budget, the subspace $\text{span}(v_1, \dots, v_r)$ identified in the MFPOD minimization \eqref{eq:MFPOD:min} recovers $V_*$.
This asymptotic convergence result provides a theoretical foundation for MFPOD, although we note that the practical motivation and benefit of MFPOD is for small computational budgets where single-fidelity POD is infeasible.

While the minimization \eqref{eq:MFPOD:min} of $\mathcal{J}_{\mymf}(V)$ is justified in expectation and in the limit $\mybudget \rightarrow \infty$ as analyzed in Section \ref{sec:analysis}, for any single draw of $\theta_1, \dots, \theta_{m_L}$ there is the practical issue that $\mathcal{C}_{\mymf}$ may have negative eigenvalues.
Because each $\lambda_j$, $1 \le j \le \dim H$, estimates a non-negative expectation $\lambda_j \approx \mathbb{E}_{\theta \sim \mu}[(u_{0}(\theta), v_j)_H^2] \ge 0$, any negative value $\lambda_j< 0$ indicates insufficient accuracy of the multifidelity estimator at $v_j$.
However, it also indicates that $v_j$ carries information for at least some of the low-fidelity models if not for the high-fidelity model itself.
To leverage this information, we introduce an adjusted, non-negative multifidelity cost function
\begin{align}
    \mathcal{J}_{\mymf}^+(V) &:= \sum_{j=1}^{\dim H} \lambda_j^+ \| \Pi_{V_{\perp}} v_j\|_H^2 
    = \sum_{j=1}^{\dim H} \lambda_j^+ (1-\| \Pi_{V} v_j\|_H^2 )
\end{align}
where $\lambda_j^+ \ge 0$ is defined via
\begin{align}\label{eq:MFPOD:lambdajplus}
    \lambda_j^+ &:= \left\{
    \begin{array}{ll}
        \lambda_j & \text{if } \lambda_j > 0 \\
        0 & \text{if } \lambda_j = 0 \text{ and } v_j \perp S \\
        \frac{1}{m_0} \sum_{i=1}^{m_0} (u_0(\theta_i), v_j)_H^2 & \text{otherwise,}
    \end{array}
    \right.
\end{align}
with $S := \text{span}\left< u_0(\theta_1), \dots, u_0(\theta_{m_0}), u_1(\theta_1), \dots, u_L(\theta_{m_L}) \right>$ being the span of all snapshots.
The adjusted cost function $\mathcal{J}_{\mymf}^+(V)$ coincides with $\mathcal{J}_{\mymf}(V)$ in all directions $v_j$ with positive curvature, while correcting the remaining directions $v_j \, \cancel{\perp} \, S$ using the Monte Carlo approximation.
Note that both the Monte Carlo and the MFMC approximations of $\mathbb{E}_{\theta \sim \mu}[(u_{0}(\theta), v_j)_H^2]$ are zero for any $v_j \perp S$ orthogonal to all snapshots.

Ordering the pairs $(\lambda_j^+, v_j) \in \mathbb{R}_{\ge 0} \times H$ such that $\lambda_1^+ \ge \lambda_2^+ \ge \dots \ge 0$, we characterize our MFPOD subspace of dimension $r$ through
\begin{align}\label{eq:MFPOD:min:plus}
    V_{\mymf} = \argmin_{V \subset H, \dim V = r} \mathcal{J}_{\mymf}^+(V).
\end{align}
This minimization is solved by $V_{\mymf} := \text{span}\{v_1, \dots, v_r\}$,
and $V_{\mymf}$ is the unique minimizer of \eqref{eq:MFPOD:min:plus} iff $\lambda_{r}^+ > \lambda_{r+1}^+$.
Analogously to single-fidelity POD, the remainder
\begin{align}
    \mathcal{J}_{\mymf}^+(V_{\mymf}) = \sum_{j=r+1}^{\dim H} \lambda_j^+ 
    \, \approx \, \mathbb{E}_{\theta \sim \mu}[\|u_{0}(\theta)-\Pi_{V_{\mymf}} u_{0}(\theta)\|_H^2]
\end{align}
can be used to determine $r$ for any target accuracy.

\section{Discrete MFPOD}
\label{sec:discrete}
Before analyzing the properties of our MFPOD method in Section \ref{sec:analysis}, we present the algorithm in discrete form. We then analyze its computational cost, discuss generalization to multiple low-fidelity models, and present an approach to adapt the control variate weights.

\subsection{MFPOD Algorithm}
To keep our exposition of the MFPOD algorithm as clear as possible, we show the two-model setting in Euclidean space $\mathbb{R}^n$ with $\mathbf{u}_{0}(\theta), \mathbf{u}_{1}(\theta) \in \mathbb{R}^n$ denoting respectively our high- and low-fidelity states for parameters $\theta \in \text{supp}(\mu)$.
For a given reduced dimension $r \ll n$ and sampling budget $\mybudget > 0$, our goal is to identify the POD basis matrix
\begin{align}\label{eq:dmfpod:Vmf}
    \mathbf{V} = \left[
    \mathbf{v}_1, \dots, \mathbf{v}_r
    \right] \in \mathbb{R}^{n \times r}, \, \mathbf{V}^{\top} \mathbf{V} = \mathbf{I}_r
\end{align}
with orthonormal columns
that minimizes the expected least-squares error
\begin{align*}
    \mathbb{E}_{\theta \sim \nu}\left[ 
    \|\mathbf{u}_{0}(\theta)-\mathbf{V}\mathbf{V}^{\top}  \mathbf{u}_{0}(\theta)\|^2
    \right] \approx \mathcal{J}_{\mymf}^+(\mathbf{V}) \ge 0
\end{align*}
approximated by the multifidelity cost function $\mathcal{J}_{\mymf}^+(\mathbf{V})$.

\begin{remark}
    In our notation from Section \ref{sec:mfpod}, we have $u_0(\theta) := \mathbf{u}_{0}(\theta) \in \mathbb{R}^n =: H$, $u_1(\theta):= \mathbf{u}_{1}(\theta) \in H$, $\Pi_V := \mathbf{V}\mathbf{V}^{\top}$, and $L=1$.
    The generalization to multiple low-fidelity models $L > 1$ is provided in Section \ref{sec:discrete:generalization}.
\end{remark}

\begin{algorithm}
\caption{Discrete MFPOD}
\label{alg:mfpod:discrete}
\begin{algorithmic}[1]
\STATE \textbf{Dependencies:} sampling function for probability distribution $\mu$, snapshot evaluation functions $\theta \mapsto \mathbf{u}_0(\theta)$ and $\theta \mapsto \mathbf{u}_1(\theta)$ of the high- and low-fidelity models
\STATE \textbf{Input:} sample sizes $1 \le m_0 < m_1$, weight $\alpha_1 \in \mathbb{R}$, cumulative energy threshold $0 < \kappa < 1$

\hrulefill
\item[]\texttt{\textbackslash \textbackslash ~ step 1: snapshots}
\STATE Draw $m_1$ i.i.d. samples $\theta_1, \dots, \theta_{m_1} \sim \mu$ \label{l:draw-samples}
\STATE Compute high-fidelity snapshot matrix
$\mathbf{S}_0 \gets \left[\mathbf{u}_0(\theta_1), \dots, \mathbf{u}_0(\theta_{m_0})\right] \in \mathbb{R}^{n \times m_0}$ \label{l:snapshots:high}
\STATE Compute low-fidelity snapshot matrices
$\mathbf{S}_1 \gets \left[\mathbf{u}_1(\theta_1), \dots, \mathbf{u}_1(\theta_{m_0})\right] \in \mathbb{R}^{n \times m_0}$ \label{l:snapshots:low}\\
and $\mathbf{S}_+ \gets \left[\mathbf{u}_1(\theta_{m_0 + 1}), \dots, \mathbf{u}_1(\theta_{m_1})\right] \in \mathbb{R}^{n \times (m_1-m_0)}$ 
\vspace{0.5em}
\item[]\texttt{\textbackslash \textbackslash ~ step 2: eigenvalues}
\STATE Set up matrix action \label{l:Cmf} \\
$\mathbf{v} \mapsto \mathbf{C}_{\mymf} \mathbf{v} := \frac{1}{m_0} \mathbf{S}_0 (\mathbf{S}_0^{\top}\mathbf{v}) + (\frac{\alpha_1}{m_1}-\frac{\alpha_1}{m_0}) \mathbf{S}_1 (\mathbf{S}_1^{\top}\mathbf{v}) + \frac{\alpha_1}{m_1} \mathbf{S}_+ (\mathbf{S}_+^{\top}\mathbf{v})$
\STATE Identify the $r_{\max} \le m_0+m_1$ eigenvalue-eigenvector pairs $(\lambda_j, \mathbf{v}_j) \in \mathbb{R} \times \mathbb{R}^n$, $\mathbf{C}_{\mymf} \mathbf{v}_j = \lambda_j \mathbf{v}_j$ for which $\lambda_j \neq 0$\label{l:eigenvalues}
\STATE \textbf{for} $j=1, \dots, r_{\max}$: 
    \textbf{if} $\lambda_j > 0$ \textbf{then} $\lambda_j^+ \gets \lambda_j$ \textbf{else} $\lambda_j^+ \gets \frac{1}{m_0} (\mathbf{S}_0^{\top}\mathbf{v}_j)^{\top} (\mathbf{S}_0^{\top}\mathbf{v}_j)$ \label{l:correction}
\vspace{0.5em}
\item[]\texttt{\textbackslash \textbackslash ~ step 3: Reduced space}
\STATE Re-order $\{(\lambda_j^+, \mathbf{v}_j)\}_{j=1}^{r_{\max}}$ such that $\lambda_1^+ \ge \lambda_2^+ \ge \dots \ge \lambda_{r_{\max}}^+ \ge 0$ \label{l:ordering}
\STATE Choose $r \le r_{\max}$ minimal with $\sum_{j=1}^{r} \lambda_j^+ \ge \kappa \sum_{j=1}^{r_{\max}} \lambda_j^+ $ \label{l:r}
\STATE Set $\mathbf{V} \gets [\mathbf{v}_1, \dots, \mathbf{v}_r]$ \label{l:V}
\vspace{0.5em}
\RETURN $\mathbf{V}$
\end{algorithmic}
\end{algorithm}

Algorithm \ref{alg:mfpod:discrete} summarizes the computational procedure of building the MFPOD subspace $\mathbf{V}_{\mymf}$ from states $\mathbf{u}_0(\theta)$ and $\mathbf{u}_1(\theta)$ of the high- and low-fidelity models for samples from the probability distribution $\mu$.
Inputs to Algorithm \ref{alg:mfpod:discrete} are high- and low-fidelity sample sizes $1 \le m_0 < m_1$, a control variate weight $\alpha_1 \in \mathbb{R}$, and a cumulative energy threshold $\kappa \in (0, 1)$ for determining the final reduced dimension.
The first step of Algorithm \ref{alg:mfpod:discrete} is to draw the samples $\theta_1, \dots, \theta_{m_1}$ (line \ref{l:draw-samples}) and compute the corresponding snapshots of the high- and low-fidelity models (lines \ref{l:snapshots:high}, \ref{l:snapshots:low}).
Note that the first $m_0$ samples $\theta_1, \dots, \theta_{m_0}$ are shared between the two models and are saved in
\begin{align}
    \mathbf{S}_0 &:= \left[\mathbf{u}_0(\theta_1), \dots, \mathbf{u}_0(\theta_{m_0})\right] \in \mathbb{R}^{n \times m_0}, &&
    \mathbf{S}_1 := \left[\mathbf{u}_1(\theta_1), \dots, \mathbf{u}_1(\theta_{m_0})\right] \in \mathbb{R}^{n \times m_0}
\end{align}
separately from the snapshot matrix
\begin{align}\label{eq:Splus}
    \mathbf{S}_+ &:= \left[\mathbf{u}_1(\theta_{m_0 + 1}), \dots, \mathbf{u}_1(\theta_{m_1})\right] \in \mathbb{R}^{n \times (m_1-m_0)}
\end{align}
containing the low-fidelity states at the additional samples $\theta_{m_0+1}, \dots, \theta_{m_1}.$

The second step of Algorithm \ref{alg:mfpod:discrete} is to identify the multifidelity modes $\{\mathbf{v}_j\}_{j=1}^{m_0+m_1}$ and their corresponding approximations $\lambda_j^+ \approx \mathbb{E}_{\theta \sim \mu}[(\mathbf{v}_j^{\top} \mathbf{u}_0(\theta))^2] \ge 0$ to their expected captured energy.
The modes $\mathbf{v}_j \in \mathbb{R}^n$ and first guesses $\lambda_j$ for $\mathbb{E}_{\theta \sim \mu}[(\mathbf{v}_j^{\top} \mathbf{u}_0(\theta))^2]$ are computed as the eigenvalue-eigenvector pairs $(\lambda_j, \mathbf{v}_j) \in \mathbb{R} \times \mathbb{R}^n$ of the multifidelity second moment matrix
\begin{align*}
    \mathbf{C}_{\mymf} 
    :=& \frac{1}{m_0} \mathbf{S}_0 \mathbf{S}_0^{\top}
    + \frac{\alpha_1}{m_1} \big[\mathbf{S}_1, \mathbf{S}_+ \big] \big[\mathbf{S}_1, \mathbf{S}_+ \big]^{\top} 
    - \frac{\alpha_1}{m_0} \mathbf{S}_1 \mathbf{S}_1^{\top} \\
    =& \frac{1}{m_0} \mathbf{S}_0 \mathbf{S}_0^{\top} + (\frac{\alpha_1}{m_1}-\frac{\alpha_1}{m_0}) \mathbf{S}_1 \mathbf{S}_1^{\top} + \frac{\alpha_1}{m_1} \mathbf{S}_+ \mathbf{S}_+^{\top}.
\end{align*}
Note that, due to its size, $\mathbf{C}_{\mymf} \in \mathbb{R}^{n \times n}$ should not be computed explicitly; rather, its low-rank structure should be exploited to compute the matrix action
\begin{align}\label{eq:matrixaction}
    \mathbf{v} \mapsto \mathbf{C}_{\mymf} \mathbf{v} :=
    \frac{1}{m_0} \mathbf{S}_0 (\mathbf{S}_0^{\top}\mathbf{v}) + \bigg(\frac{\alpha_1}{m_1}-\frac{\alpha_1}{m_0}\bigg) \mathbf{S}_1 (\mathbf{S}_1^{\top}\mathbf{v}) + \frac{\alpha_1}{m_1} \mathbf{S}_+ (\mathbf{S}_+^{\top}\mathbf{v}) \in \mathbb{R}^{n}
\end{align}
in $\mathcal{O}(n(m_0+m_1))$ flops.
The eigenvalue problem $\mathbf{C}_{\mymf} \mathbf{v}_j = \lambda_j \mathbf{v}_j$ can then be solved efficiently using iterative solvers (for example, the Lanczos algorithm).
Because of the low-rank structure, $\mathbf{C}_{\mymf}$ can have at most $m_0 + m_1$ non-zero eigenvalues, and their corresponding eigenvectors are in the column-span of the stacked snapshot matrix $\mathbf{S} := [\mathbf{S}_0, \mathbf{S}_1, \mathbf{S}_+] \in \mathbb{R}^{n \times m_0 + m_1}$.
For iterative eigenvalue solvers that employ a form of power iteration, such as Lanczos algorithm, initialization at $\mathbf{S}$ guarantees that all non-zero eigenvalues are identified.

With the multifidelity modes $\{\mathbf{v}_j\}_{j=1}^{m_0+m_1}$ determined, the for-loop in line \ref{l:correction} identifies all instances where the multifidelity approximation $\lambda_j \le 0$ might fail to approximate $\mathbb{E}_{\theta \sim \mu}[(\mathbf{v}_j^{\top}\mathbf{u}_0(\theta))^2] \ge 0$.
In these cases, we default to the Monte Carlo approximation
\begin{align*}
    \mathbb{E}_{\theta \sim \mu}[(\mathbf{v}_j^{\top}\mathbf{u}_0(\theta))^2] \approx \frac{1}{m_0} (\mathbf{S}_0^{\top}\mathbf{v}_j)^{\top} (\mathbf{S}_0^{\top}\mathbf{v}_j) \ge 0.
\end{align*}
Going into the third and final step of Algorithm \ref{alg:mfpod:discrete}, each $\lambda_j^+$ then contains non-negative approximations to $\mathbb{E}_{\theta \sim \mu}[(\mathbf{v}_j^{\top}\mathbf{u}_0(\theta))^2]$ as defined \eqref{eq:MFPOD:lambdajplus}.
After re-ordering by magnitude, we choose the reduced dimension $r$ in line \ref{l:r} based on the typical POD cumulative energy criterion
\begin{align}\label{eq:energy}
    \kappa \le \frac{\sum_{j=1}^r \lambda_j^+}{\sum_{j=1}^{m_0 + m_1} \lambda_j^+} \approx \frac{\mathbb{E}_{\theta \sim \mu}[\|\mathbf{V}\mathbf{V}^{\top} \mathbf{u}_0(\theta)\|^2]}{\mathbb{E}_{\theta \sim \mu}[\| \mathbf{u}_0(\theta)\|^2]}
\end{align}
and set $\mathbf{V} := [\mathbf{v}_1, \dots, \mathbf{v}_r] \in \mathbb{R}^{n \times r}$ to contain the first $r$ modes as columns.
Of course, other choices of $r$ are admissible.

\subsection{Computational cost}
Algorithm \ref{alg:mfpod:discrete} has three major sources of computational costs: The sampling costs in lines \ref{l:draw-samples} to \ref{l:snapshots:low}, the solve of the eigenvalue problem $\mathbf{C}_{\mymf} \mathbf{v}_j = \lambda_j \mathbf{v}_j$ for $m_0 + m_1$ different eigenvalue-eigenvector pairs in line \ref{l:eigenvalues}, and the Monte Carlo corrections needed for $\lambda_j^+$ to be non-negative in line~\ref{l:correction}.
All remaining computations are independent from the high-fidelity dimension $n$.

The sampling cost is determined by the sample sizes $m_0$ and $m_1$, and accumulates to $m_0 c_0 + m_1 c_1$ where $c_0 \ge c_1 > 0$ are the cost of evaluating the high- and low-fidelity model respectively.
For $1 \le m_0 < m_1$ to hold, we require a snapshot sampling budget of at least $\mybudget \ge c_0 + 2c_1$. For large-scale scientific applications, computing the snapshots is the dominant cost.

When exploiting the low-rank structure of $\mathbf{C}_{\mymf}$ via the matrix action \eqref{eq:matrixaction},
using the Lanczos algorithm to solve the eigenvalue problem in line \ref{l:eigenvalues} for the $m_0 + m_1$ largest (by absolute value) eigenvalues and eigenvectors scales as $\mathcal{O}(k_{\rm{itr}}n(m_0+m_1)^2)$, where $k_{\rm{itr}}$ is the number of iterations and depends on the spectrum of $\mathbf{C}_{\mymf}$.
The cost $\mathcal{O}(k_{\le 0}m_0n)$ of correcting $k_{\le 0}$ non-positive multifidelity eigenvalues $\lambda_j \le 0$ to the Monte Carlo approximation $\frac{1}{m_0} (\mathbf{S}_0^{\top}\mathbf{v}_j)^{\top} (\mathbf{S}_0^{\top}\mathbf{v}_j)$ is negligible in comparison.

\subsection{Generalization to multiple low-fidelity models}\label{sec:discrete:generalization}
The generalization of Algorithm \ref{alg:mfpod:discrete} to the case of multiple low-fidelity models is straightforward:
For each low-fidelity model $\ell$, with sample size $m_{\ell} > m_{\ell-1}$, we compute the snapshot matrices
\begin{align}
    \mathbf{S}_{\ell} &:= \left[\mathbf{u}_{\ell}(\theta_1), \dots, \mathbf{u}_{\ell}(\theta_{m_{\ell-1}})\right] \in \mathbb{R}^{n \times m_{\ell-1}}, \\
    \mathbf{S}_{\ell, +} &:= \left[\mathbf{u}_{\ell}(\theta_{m_{\ell-1}+1}), \dots, \mathbf{u}_{\ell}(\theta_{m_{\ell}})\right] \in \mathbb{R}^{n \times (m_{\ell} - m_{\ell-1})}.
\end{align}
For model weights $\alpha_1, \dots, \alpha_L \in \mathbb{R}$, the multifidelity covariance matrix $\mathbf{C}_{\mymf}$ is then defined as
\begin{align}\label{eq:discrete:Cmf:multiple}
    \mathbf{C}_{\mymf} 
    =& \frac{1}{m_0} \mathbf{S}_0 \mathbf{S}_0^{\top} 
    + \sum_{\ell=1}^L \left( (\frac{\alpha_{\ell}}{m_{\ell}}-\frac{\alpha_{\ell}}{m_{\ell-1}}) \mathbf{S}_{\ell} \mathbf{S}_{\ell}^{\top} + \frac{\alpha_{\ell}}{m_{\ell}} \mathbf{S}_{\ell, +} \mathbf{S}_{\ell, +}^{\top} \right).
\end{align}
All remaining steps of Algorithm \ref{alg:mfpod:discrete} remain the same.

\subsection{Adaptive weights}

\begin{algorithm}
\caption{Adaptive MFPOD weights}
\label{alg:adaptive}
\begin{algorithmic}
\STATE \textbf{Input:} Snapshot matrices $\mathbf{S}_0, \mathbf{S}_1 \in \mathbb{R}^{n \times m_0}$, $\mathbf{S}_+ \in \mathbb{R}^{n \times (m_1-m_0)}$ \\
\hrulefill

\STATE Initialize $\mathbf{W} \gets \mathbf{0} \in \mathbb{R}^{n \times 0}$, $j \gets 1$, $r_{\max} \gets 0$
\vspace{0.5em}
\WHILE{$\|\mathbf{S}_0-\mathbf{W}\mathbf{W}^{\top}\mathbf{S}_0\| > 0$}
\item[]\texttt{\textbackslash \textbackslash ~ update weight}
    \STATE \textbf{for} $i=1,\ldots,m_0 $ \textbf{:} $x_i \gets \|\mathbf{u}_0(\theta_i)-\mathbf{W}\mathbf{W}^{\top}\mathbf{u}_0(\theta_i)\|^2$, $y_i \gets \|\mathbf{u}_1(\theta_i)-\mathbf{W}\mathbf{W}^{\top}\mathbf{u}_1(\theta_i)\|^2$
    \vspace{-1em}
    \STATE Compute sample variance $\sigma_{1}^2$ of $\{y_i\}_{i=1}^{m_0}$ 
    \STATE Compute sample covariance $\sigma_{0, 1}$ between $\{x_i\}_{i=1}^{m_0}$ and $\{y_i\}_{i=1}^{m_0}$
    \STATE \textbf{if} $\sigma_1^2 > 0$, set $\alpha_1 \gets \sigma_{0,1} / \sigma_1^2$ \textbf{else} $\alpha_1 \gets 0$
    \vspace{0.5em}
    \item[]\texttt{\textbackslash \textbackslash ~ most important mode}
    \STATE Set up matrix action \\
$\mathbf{v} \mapsto \mathbf{C}_{\mymf} \mathbf{v} := \frac{1}{m_0} \mathbf{S}_0 (\mathbf{S}_0^{\top}\mathbf{v}) + (\frac{\alpha_1}{m_1}-\frac{\alpha_1}{m_0}) \mathbf{S}_1 (\mathbf{S}_1^{\top}\mathbf{v}) + \frac{\alpha_1}{m_1} \mathbf{S}_+ (\mathbf{S}_+^{\top}\mathbf{v})$
    \STATE Compute eigenvalue-eigenvector pair $(\lambda_j, \mathbf{v}_j) \in \mathbb{R} \times \mathbb{R}^{n}$, $\mathbf{C}_{\mymf}\mathbf{v}_j = \lambda_j$, such that $|\lambda_j|$ is maximal with $\|\mathbf{v}_j\|=1$, $\mathbf{v}_j \perp \mathbf{W}$
\item \textbf{if} $\lambda_j > 0$ set $\lambda_j^+ \gets \lambda_j$ \textbf{else} set $\lambda_j^+ \gets \frac{1}{m_0} (\mathbf{S}_0 \mathbf{v}_j)^{\top}\mathbf{S}_0 \mathbf{v}_j \ge 0$
\vspace{0.5em}
\STATE Expand $\mathbf{W} \gets [\mathbf{W}, \mathbf{v}_j]$, $j \gets j+1$, $r_{\max} \gets r_{\max}+1$
\ENDWHILE
\RETURN $\{(\lambda_j^+, \mathbf{v}_j)\}_{j=1}^{r_{\max}}$
\end{algorithmic}
\end{algorithm}

Finally, we address the control variate weight $\alpha_1$.
Considering that the final reduced dimension $r$ is selected based on the energy criterion \eqref{eq:energy}, choosing $\alpha_1 = \alpha_1^*(\mathbf{0}) = \frac{\text{Cov}_{\theta \sim \mu} \left(\|\mathbf{u}_{0}(\theta)\|^2, \|\mathbf{u}_{1}(\theta)\|^2 \right)}{\mathbb{V}_{\theta \sim \mu} \left(  \|\mathbf{u}_{1}(\theta)\|^2 \right)}$ is optimal to minimize the MSE of $\mathcal{J}_{\mymf}(\mathbf{0}) \approx \mathbb{E}_{\theta \sim \mu}[\|\mathbf{u}_0(\theta)\|^2]$
as an estimator of the total energy.
The optimal weight $\alpha_1^*(\mathbf{0})$ can be approximated using the sample variance and covariance of $\{\|\mathbf{u}_0(\theta_i)\|^2\}_{i=1}^{m_0}$, $\{\|\mathbf{u}_1(\theta_i)\|^2\}_{i=1}^{m_0}$ on the shared training parameters.
When the high- and low-fidelity model share similar dominating features, $\alpha^*(\mathbf{0}) \approx 1$.

Because the control variate weight $\alpha_1$ balances the contributions of the high- and low-fidelity model in the multifidelity cost function $\mathcal{J}_{\mymf}$, a large $|\alpha_1|$ risks prioritizing low-fidelity features in the higher-order MFPOD modes.
Algorithm \ref{alg:adaptive} addresses this issue by adjusting the weight $\alpha_1$ throughout an iterative selection of the pairs $(\lambda_j^+, \mathbf{v}_j)$.
It can be integrated into Algorithm \ref{alg:mfpod:discrete} as an extension by replacing lines \ref{l:Cmf} to \ref{l:correction} with the call 
\begin{align*}
    \{(\lambda_k^+, \mathbf{v}_j)\}_{j=1}^{r_{\max}} \gets \texttt{AdaptiveWeights}(\mathbf{S}_0, \mathbf{S}_1, \mathbf{S}_+).
\end{align*}
For any integer $j \ge 0$, $(\lambda_j^+, \mathbf{v}_j)$ is chosen as the dominant eigenvalue-eigenvector pair for the cost function
\begin{align*}
    \mathcal{J}_{\mymf}([\mathbf{W}, \mathbf{v}])
    \approx \mathbb{E}_{\theta \sim \mu}[\|\mathbf{u}_0(\theta) - \mathbf{W}\mathbf{W}^{\top}\mathbf{u}_0(\theta) - \mathbf{v}\mathbf{v}^{\top}\mathbf{u}_0(\theta)\|^2], 
\end{align*}
where $\mathbf{W}=[\mathbf{v}_1, \dots, \mathbf{v}_{j-1}]$ contains all modes selected previously and $\mathcal{J}_{\mymf}$ is computed with $\alpha_1 \approx \alpha_1^*(\mathbf{W})$.

\section{Analysis}
\label{sec:analysis}

In this section we show that the MFPOD operator $\mathcal{C}_{\mymf}$ converges in probability to the POD second moment operator $\mathcal{C}$
in the Hilbert-Schmidt norm.
The rate of convergence is inversely proportional to the  snapshot sampling cost $\mybudget = \sum_{\ell=0}^{L}m_{\ell}c_{\ell}$.
We derive an upper bound for the MSE of the summed MFPOD eigenvalues $\sum_{j=1}^r \lambda_j$ approximating $\sum_{j=1}^r \lambda_j^*$, and, under the condition that $\lambda_r^* > \lambda_{r+1}^*$, for the summed projection error $\sum_{j=1}^r\|v_j-\Pi_{V_r^*}v_j\|_H^2$ of the MFPOD eigenfunctions $v_1, \dots, v_r$ onto $V^* := \text{span}(v_1^*, \dots, v_r^*)$.
Both MSE upper bounds are of order $\mathcal{O}(r / \mybudget)$, showing that, in the limit of arbitrarily large budget, the MFPOD method recovers both the truth POD subspace $V^*$ and its captured energy in probability.

We start with the convergence of $\mathcal{C}_{\mymf}$ to $\mathcal{C}$ in probability in the Hilbert-Schmidt norm $\|\mathcal{C}_{\mymf}-\mathcal{C}\|_{\rm{HS}}^2 := \sum_{j=1}^{\dim H} \|(\mathcal{C}_{\mymf}-\mathcal{C}) v_j^*\|_H^2$.

\begin{theorem}\label{thm:convergence:C}
    Let $(H, (\, . \,, \, . \, )_H)$ be a separable Hilbert space.
    Let $\theta : \Omega \rightarrow \mathbb{R}$ be a random variable on a probability space $(\Omega, \mathcal{F}, \mathbb{P})$ with distribution $\mu$.
    Let $L \ge 1$ be the number of available surrogate models.
    Assume $\mathbb{E}[\|u_{\ell}(\theta)\|^4_H] < \infty$ for $\ell = 0, \dots, L$.
    Let $\alpha_1, \dots, \alpha_L \in \mathbb{R}$ be fixed weights and $1 =: q_0 < q_1 < \dots < q_L$ be fixed integers.
    For any $m_0 \in \mathbb{N}$, define $\mathcal{C}_{\mymf}(m_0) : H \rightarrow H$ as in \eqref{eq:Cmf} for the sample sizes $m_{\ell} := q_{\ell} m_0 \in \mathbb{N}$, $\ell = 1, \dots, L$.
    Define $\mathcal{C}$ as in \eqref{eq:C}.
    Then there exists $\gamma < \infty$ independent of $m_0$ with
\begin{align}\label{eq:convergence:C}
        \mathbb{E}[\|\mathcal{C}_{\mymf}(m_0)-\mathcal{C}\|_{\rm{HS}}^2] = \frac{\gamma}{m_0}
    \end{align}
    for all $m_0 \in \mathbb{N}$.
    In particular, $\|\mathcal{C}_{\mymf}(m_0)-\mathcal{C}\|_{\rm{HS}}^2 \xrightarrow[]{P} 0$ for $m_0 \rightarrow \infty$.
\end{theorem}

\begin{proof}
    The proof is provided in Appendix~\ref{sec:appendix}.
\end{proof}

Theorem \ref{thm:convergence:C} establishes that the MFPOD operator $\mathcal{C}_{\mymf}$ converges to $\mathcal{C}$ (in the Hilbert-Schmidt norm) in probability for $m_0 \rightarrow \infty$.
Because the sampling ratios $q_{\ell} = m_{\ell} / m_{0}$, $\ell = 1, \dots, L$, are fixed, $m_0$ is proportional to the invested computational budget $\mybudget = \sum_{\ell=0}^{L}m_{\ell} c_{\ell} = m_0 \sum_{\ell=0}^{L}q_{\ell} c_{\ell}$, and thus $\mathbb{E}[\|\mathcal{C}_{\mymf}-\mathcal{C}\|_{\rm{HS}}^2] = \frac{\gamma}{\mybudget}\sum_{\ell=0}^{L}q_{\ell} c_{\ell} \in \mathcal{O}(\mybudget^{-1})$.
The constant $\gamma$ depends on the models $u_0(\theta), \dots, u_L(\theta)$ and the sampling ratios $q_1, \dots, q_L$, but is independent of $m_0$. 
To show that $\gamma < \infty$, the proof of Theorem \ref{thm:convergence:C} uses the assumption that $\mathbb{E}[\|u_{\ell}(\theta)\|_H^4] < \infty$ for $\ell = 0, \dots, L$, which is a common technical assumption that holds, for example, when $\|u_{\ell}(\theta)\|_H$ is uniformly bounded for almost all $\theta \in \text{supp}(\mu)$.

Building upon Theorem \ref{thm:convergence:C}, the following proposition establishes upper bounds on the MSEs of the summed MFPOD eigenvalues and on the projection error of the MFPOD modes $v_1, \dots, v_r$ onto the truth POD subspace $V^*$.
In particular, using that the budget $\mybudget = m_0 \sum_{\ell=0}^{L}q_{\ell} c_{\ell}$ is proportional to $m_0$, Proposition \ref{thm:convergence:lambda} yields the convergence rates \eqref{eq:bound:sec2:eigvals} and \eqref{eq:bound:sec2:eigvecs} discussed in Section \ref{sec:MFPOD-opt}.

\begin{proposition}\label{thm:convergence:lambda}
    Under the conditions of Theorem \ref{thm:convergence:C}, let $(\lambda_j^*, v_j^*) \in \mathbb{R} \times H$, $j = 1, \dots, \dim H$, be the eigenvalue-eigenfunction pairs of $\mathcal{C}$ such that $\lambda_1^* \ge \lambda_2^* \ge \dots \ge 0$, and $\{v_j^*\}_{j=1}^{\dim H}$ is an orthonormal basis of $H$.
    Let $r < \dim H$ be a fixed integer.
    For $m_0 \in \mathbb{N}$ let $(\lambda_j, v_j) \in \mathbb{R} \times H$, $j=1, \dots, \dim H$, be eigenvalue-eigenfunction pairs of $\mathcal{C}_{\mymf}(m_0)$ such that $\lambda_1 \ge \dots \ge \lambda_r \ge \sup_{j > r}\lambda_j$ and $v_1, \dots, v_r$ are orthonormal.
    Then
\begin{align}\label{eq:bound:eigvals}
        \mathbb{E}\left[ \bigg(\sum_{j=1}^r \lambda_j -\sum_{j=1}^r \lambda_j^*\bigg)^2 \right] \le \frac{r \gamma}{m_0}
    \end{align}
    where $\gamma$ is the same constant as in Theorem \ref{thm:convergence:C}.
    If $\lambda_r^* > \lambda_{r+1}^*$, then also
    \begin{align}\label{eq:bound:eigvecs}
        \mathbb{E}\hspace{-0.2em}\left[\hspace{-0.2em} \bigg(\hspace{-0.2em}\sum_{j=1}^r \|v_j^* - \Pi_{V}v_j^*\|_H^2\hspace{-0.2em}\bigg)^2
        \right] \hspace{-0.2em}
        =\mathbb{E}\hspace{-0.2em}\left[\hspace{-0.2em} \bigg(\hspace{-0.2em}\sum_{j=1}^r \|v_j - \Pi_{V_*}v_j\|_H^2\hspace{-0.2em}\bigg)^2\hspace{-0.2em}
        \right] \le \frac{2r\gamma}{m_0(\lambda_r^* - \lambda_{r+1}^*)^2}
    \end{align}
    with $V := \text{span}(v_1, \dots, v_r)$ and $V_* := \text{span}(v_1^*, \dots, v_r^*)$.
\end{proposition}
\begin{proof}
    The proof is provided in Appendix~\ref{sec:appendix}.
\end{proof}

The first bound \eqref{eq:bound:eigvals} in Proposition \ref{thm:convergence:lambda} justifies the captured energy criterion \eqref{eq:energy} for selecting the reduced dimension $r$.
The bound shows that $\sum_{j=1}^r \lambda_j \xrightarrow[]{P} \sum_{j=1}^r \lambda_j^*$ as $m_0 \rightarrow \infty$.
Using that $\mathcal{J}_{\mymf}(\{0\}) = \sum_{j=1}^{\dim H} \lambda_j$ by Lemma \ref{thm:costfunction}, $\mathbb{E}[\mathcal{J}_{\mymf}(\{0\})] = \sum_{j=1}^{\dim H} \lambda_j^* > 0$ by \eqref{eq:MFPOD:unbiased}, and $\mathbb{V}[\mathcal{J}_{\mymf}(\{0\})] \in \mathcal{O}(m_0^{-1})$ by \eqref{eq:MFPOD:MSE}, the continuous mapping theorem yields that the MFPOD captured energy converges in probability 
\begin{align*}
    \frac{\sum_{j=1}^r \lambda_j}{\sum_{j=1}^{\dim H} \lambda_j}
    &\xrightarrow[]{P}\frac{\sum_{j=1}^r \lambda_j^*}{\sum_{j=1}^{\dim H} \lambda_j^*} = \frac{\mathbb{E}_{\theta \sim \mu}[\, \|\Pi_{V_*} u_0(\theta)\|_H^2]}{\mathbb{E}_{\theta \sim \mu}[\, \|u_0(\theta)\|_H^2]} 
    & \text{for } m_0 \rightarrow \infty
\end{align*}
to the maximum energy that can be retained of $\mathbb{E}_{\theta \sim \mu}[\, \|u_0(\theta)\|_H^2]$ by an $r$-dimensional linear subspace of $H$.
Consequently, for a large enough computational budget $\mybudget$ the captured energy criterion $\eqref{eq:energy}$ is unlikely to mislead the choice of reduced dimension.

The second bound \eqref{eq:bound:eigvecs} in Proposition \ref{thm:convergence:lambda} 
shows that as the computational budget $\mybudget$ increases the MFPOD subspace $V=\text{span}(v_1, \dots, v_r)$ aligns with the optimal subspace $V_* = \text{span}(v_1^*, \dots, v_r^*)$.
Geometrically, $\sum_{j=1}^r\|v_j^* - \Pi_{V}v_j^*\|_H^2 = \sum_{j=1}^r\sin^2(\beta_j)$ is the sum over the squared sines of the principal angles $\beta_1, \dots, \beta_r$ between $V$ and $V_*$.
By \eqref{eq:bound:eigvecs}, $0 \le  \sum_{j=1}^r\sin^2(\beta_j) \xrightarrow[]{P} 0$ for $m_0 \rightarrow \infty$, and thus each principal angle $\beta_1, \dots, \beta_r$ becomes arbitrarily small (in probability) for increasing budget.
The upper bound \eqref{eq:bound:eigvecs} scales inversely to the squared spectral gap $(\lambda_r^* - \lambda_{r+1}^*)^2$ which is an indicator for how easily $v_r^*$ and $v_{r+1}^*$ can be distinguished.

\begin{remark}\label{rmk:convergence:C}
    In the discrete setting of Section \ref{sec:discrete}, the Hilbert-Schmidt norm simplifies to the Frobenius norm, such that \eqref{eq:convergence:C} states $\mathbb{E}[\|\mathbf{C}_{\mymf}-\mathbf{C}_*\|_F^2] = \gamma / m_0$, where $\mathbf{C}_{\mymf}$ is defined as in \eqref{eq:discrete:Cmf:multiple} and $\mathbf{C}_* := \mathbb{E}_{\theta \sim \mu}[\mathbf{u}_0(\theta)\mathbf{u}_0(\theta)^{\top}] \in \mathbb{R}^{n \times n}$.
    Letting $(\lambda_1^*, \mathbf{v}_1^*), \dots, (\lambda_n^*, \mathbf{v}_n^*) \in \mathbb{R}_{\ge 0} \times \mathbb{R}^n$ be the eigenvalue-eigenvector pairs of $\mathbf{C}_*$ such that $\lambda_1^* \ge \dots \ge \lambda_n^* \ge 0$ and $\mathbf{v}_1, \dots, \mathbf{v}_n$ are orthonormal, \eqref{eq:bound:eigvecs} states that $\mathbb{E}[\|\mathbf{V}_* - \mathbf{V}\mathbf{V}^{\top}\mathbf{V}_*\|_F^4] \le \frac{2r\gamma}{m_0(\lambda_r^*-\lambda_{r+1}^*)^2}$ with $\mathbf{V}_* := [\mathbf{v}_1^*, \dots, \mathbf{v}_r^*] \in \mathbb{R}^{n \times r}$ and the columns of $\mathbf{V} = [\mathbf{v}_1, \dots, \mathbf{v}_r] \in \mathbb{R}^{n \times r}$ are the eigenvectors of $\mathbf{C}_{\mymf}$ for the largest $r$ eigenvalues.
\end{remark}

\section{Numerical results} \label{sec:results}
We present results for an illustrative example of the parameterized steady-state advection-diffusion equation and a large-scale ice sheet dynamics model.

\subsection{Advection-diffusion equation}
\label{sec:results:1d}

We consider the parameterized steady-state advection-diffusion equation
\begin{align}\label{eq:1d}
    -\frac{1}{\theta} \Delta u - u_x = 0 \qquad \text{on } D := (0,1)
\end{align}
with uniformly distributed parameter $\theta \in \Theta := [\theta_{\min}, \theta_{\max}] :=[1, 100]$, $\theta \sim \mu = \mathcal{U}([\theta_{\min}, \theta_{\max}])$ and Dirichlet boundary conditions $u(0) = 1$, $u(1) = 0$.
The exact solution for any given $\theta \in \Theta$ at $x \in D$ is $u(\theta)(x) = -e^{-\theta x}(e^{-\theta}-1)^{-1} + (1-e^{\theta})^{-1}$.
We discretize \eqref{eq:1d} with linear finite elements on uniform meshes over $D$, with $n=4097$ degrees of freedom for our high-fidelity model and $n_1 = 33$ degrees of freedom for our low-fidelity model.
We non-dimensionalize the model costs to $c_0 := 1$ for the high-fidelity model, and, because the compute time of our finite element solver scales linearly in the degrees of freedom, $c_1 := \frac{33}{4097} \approx 0.008$ for the low-fidelity model.
For single-fidelity POD, an integer budget $\mybudget \in \mathbb{N}$ then corresponds directly to the number of permitted samples from the high-fidelity model.
For MFPOD, we split the budget approximately evenly between the two models with $m_0 := \lfloor \mybudget / 2 \rfloor$ high-fidelity and $m_1 := 124 m_0 = \lfloor c_0 / c_1 \rfloor m_0$ low-fidelity samples; in particular, $m_0 = 2$ and $m_1 = 248$ for $\mybudget=5$, and $m_0 =5$ and $m_1 = 620$ for $\mybudget = 10$.

As a reference ground truth, we define the reference solution as the POD eigenvalues, $\{\lambda_i^{\rm{ref}}\}_{i=1}^{100,000}$, and eigenmodes, $V_{\rm{ref}}$, from $100,000$ high-fidelity snapshots $\{u_0(\theta_i^{\rm{ref}})\}_{i=1}^{100,000}$, where the parameters $\{\theta_i^{\rm{ref}}\}_{i=1}^{100,000}$ are spaced equidistantly on $\Theta$.

Figure~\ref{fig:1D:eigvals} plots the eigenvalues associated with the single-fidelity POD and MFPOD bases, computed respectively from \eqref{eq:POD:min:eigvals} and \eqref{eq:MFPOD:lambdajplus} in the $L^2(D)$ inner product. The figure shows the dominant spectra of the resulting eigenvalues for $\mybudget=5$ and $\mybudget=10$, with $100$ i.i.d.\ draws of samples in each case. 
Setting a tolerance on the estimated eigenvalues of $10^{-10}$, the figure shows that MFPOD identifies a larger number of modes with non-zero eigenvalues than POD, and that the MFPOD eigenvalues decay more slowly than the POD eigenvalues. This is an indication that, for a given budget, the multifidelity snapshot set is more information-rich than the single-fidelity snapshot set. For $\mybudget = 5$, POD samples five high-fidelity snapshots and 
identifies up to\footnote{In three of the 100 repeats, the fourth eigenvalue was below $10^{-10}$, and in 29 of the 100 repeats the fifth eigenvalue was below $10^{-10}$.} five eigenvalues above $10^{-10}$.
For the same budget, MFPOD samples two high-fidelity and $248$ low-fidelity snapshots, and identifies up to ten eigenvalues above $10^{-10}$, with 97 of the 100 draws identifying at least six modes.
We also note that for the first few dominant modes, the spread of the estimated eigenvalue is lower for MFPOD than for POD, reflecting the variance reduction expected from the MFPOD theory.
This variance reduction translates into the practical benefit that, for a low snapshot budget, MFPOD is less susceptible to the variation in modes induced by the randomness of the snapshot draws.

\begin{figure}
    \centering
    \includegraphics[width=\linewidth]{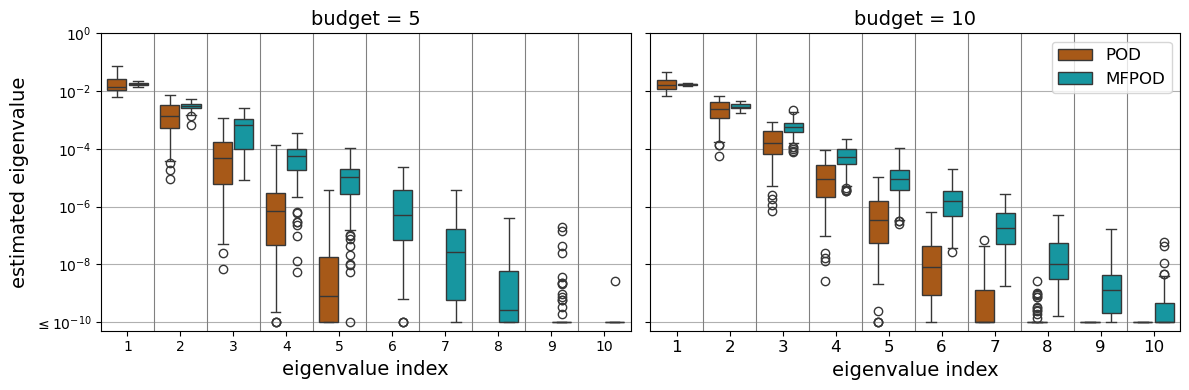}
    \caption{Distribution of estimated eigenvalues for the advection-diffusion example using single-fidelity and multifidelity POD for a computational budget of $\mybudget = 5$ (left) and $\mybudget = 10$ (right).
    The distributions are over 100 i.i.d.\ repeats of the same experimental setup.
    }
    \label{fig:1D:eigvals}
\end{figure}

We now assess the quality of the MFPOD basis compared with the single-fidelity POD basis and the reference solution. For any reduced dimension $r$, we obtain the $r$-dimensional single- and multifidelity POD spaces $V_0$ and $V_{\mymf}$ through \eqref{eq:POD:min} and \eqref{eq:MFPOD:min:plus}, respectively.
To assess the quality of a basis $V$, we define $E(V)$ to be the captured energy relative to the reference snapshot set:
\begin{align}
    E(V) := 100
    \frac{\sum_{i=1}^{100,000} \|\Pi_V u_0(\theta_i^{\rm{ref}})\|_{L^2(D)}^2}{\sum_{i=1}^{100,000} \|u_0(\theta_i^{\rm{ref}})\|_{L^2(D)}^2} \,[\%],
\end{align}
which provides a measure of how accurately the reference snapshot set can be represented in the basis $V$.
The reference basis $V_{\rm{ref}}$ achieves the maximal values of $E$ possible for a given reduced dimension. These maximal values are shown for reduced dimensions from $r=1$ to $r=8$ in 
Figure~\ref{fig:1D:energy} by the horizontal lines marked `reference.'
Figure~\ref{fig:1D:energy} plots the energy captured by single-fidelity POD, $E(V_0)$, and by MFPOD, $E(V_{\mymf})$, as a function of the reduced dimension $r$.
In almost all cases, the MFPOD subspaces capture more energy than the corresponding high-fidelity POD subspaces, indicating a higher quality basis.
We note that this improvement is achieved even without tuning MFPOD to have an optimal allocation between high- and low-fidelity snapshots. 
It is notable that in the lower budget case ($\mybudget = 5$), MFPOD is able to produce useful basis vectors (i.e., those that contribute in a meaningful way to the captured energy) while even going beyond the reduced dimension that is possible for high-fidelity POD. 
There is also a notable reduction in the variance of the captured energy for the MFPOD results, resulting in a more consistent performance over snapshot draws.

One might ask: why not just use the low-fidelity model alone?
Figure \ref{fig:1D:energy} also plots the captured energy $E(V_1)$ for POD spaces $V_1$ trained on $m_1 = \lfloor c_0 / c_1 \rfloor \mybudget = 124 \mybudget$ low-fidelity snapshots (and no high-fidelity snapshots).
Due to the large sample size, the low-fidelity POD spaces perform consistently and show almost no variation over the 100 repeats.
For the smaller reduced dimensions, the
captured energy $E(V_1)$ is similar to $E(V_{\mymf})$ with MFPOD and better than $E(V_0)$ with high-fidelity POD;
however, as the reduced dimension increases, the captured energy $E(V_1)$ with the low-fidelity POD spaces stagnates 
due to the bias in the low-fidelity model. (In this example, the bias is due to the discretization error associated with the coarse grid.) In contrast, the captured energy continues to increase for the MFPOD spaces, highlighting that $V_{\mymf}$ exploits the benefits of the low-fidelity model's low cost while avoiding the limitations of its bias through the multifidelity control-variate formulation.

\begin{figure}
    \centering
    \includegraphics[width=\linewidth]{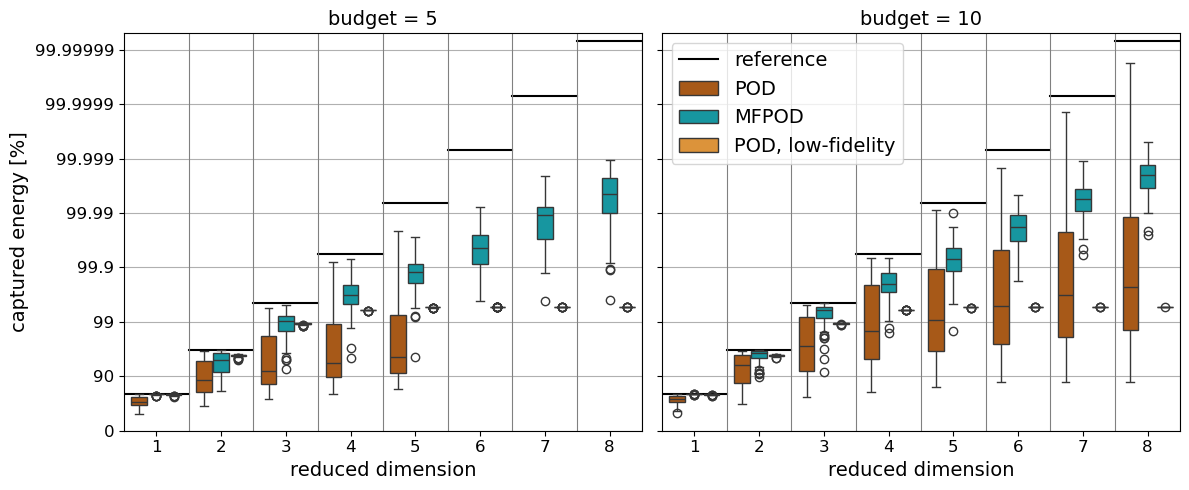}
    \caption{Captured energy relative to the reference snapshot set, for single-fidelity and multifidelity POD spaces of different reduced dimensions.
    Note that the variations in POD with low-fidelity snapshots is comparatively small due to the large low-fidelity sample size ($m_1 = 124 \mybudget$) such that the corresponding box plot has collapsed.
    }
    \label{fig:1D:energy}
\end{figure}

\begin{figure}
    \centering
    \includegraphics[width=\linewidth]{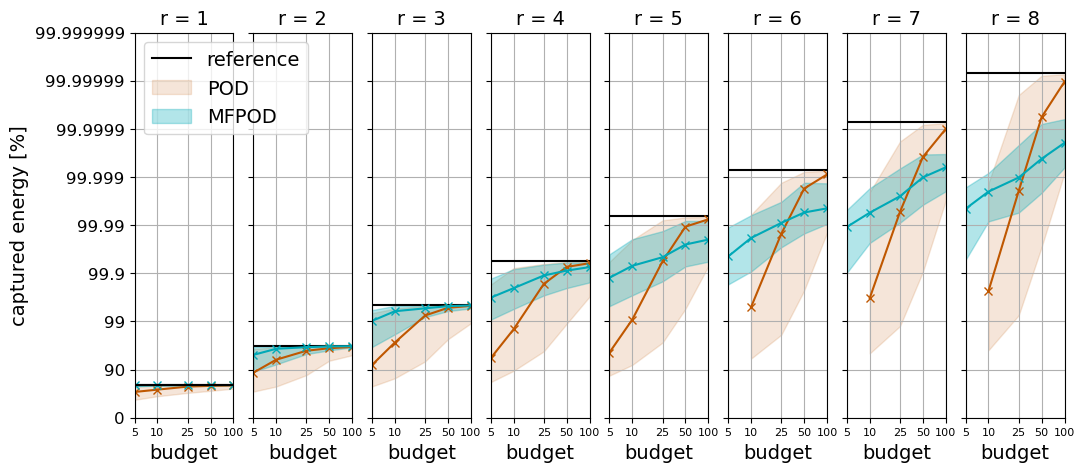}
    \caption{Captured energy by the single-fidelity and multifidelity POD spaces for increasing computational budgets.
    The shaded regions mark the $5\%$- to $95\%$-percentiles, the median is marked with $\times$.
    }
    \label{fig:1D:by-budget}
\end{figure}

To further compare the effect of the computational budget on POD and MFPOD, we additionally compute their modes for budgets $\mybudget \in \{20, 50, 100\}$.
Figure~\ref{fig:1D:by-budget} compares the captured energy by the POD and MFPOD subspaces by reduced dimension and budget.
MFPOD clearly outperforms POD for the smaller computational budgets $\mybudget \le 20$ regardless of reduced dimension.
However, as budget and reduced dimension increase, POD improves faster than MFPOD and in median captures more energy once $r \ge 5$ and $\mybudget \ge 50$.
Still, even for $r \ge 5$ and $\mybudget \ge 50$, MFPOD performs more consistently with a smaller spread between its $5$th- to $95$th-percentiles, such that its worst-case performance exceeds the worst-case performance of POD.
Further improvements in MFPOD could be achieved using a strategy to optimally allocate samples across fidelity levels, particularly for higher budgets.

\subsection{Pine Island glacier}
\label{sec:results:pig}

Our second example considers ice sheet dynamics using the Ice-sheet and Sea-level System Model (ISSM, \cite{larour2012continental}). Specifically, we model the ice velocities $v_x, v_y : D \rightarrow \mathbb{R}$ of the Pine Island glacier in Western Antarctica.
Our three-dimensional spatial domain $D$ is uniquely defined by its horizontal extent $D_{2D} \subset \mathbb{R}^2$, its surface altitude $s : D_{2D} \rightarrow \mathbb{R}_{\ge 0}$, and its ice thickness $h : D_{2D} \rightarrow \mathbb{R}_{\ge 0}$, through
\begin{align*}
    D = \{(x, y, z)^{\top} \in \mathbb{R}^3:\,
    (x, y)^{\top} \in D_{2D}, \, 
    s(x,y) > z > s(x, y)-h(x, y)
    \}.
\end{align*}
Figure~\ref{fig:PIG:surface} depicts $D_{2D}$, which employs the domain outline from \cite{seroussi2014sensitivity}. 
The fields $s$ and $h$ are also plotted in Figure~\ref{fig:PIG:surface}.

\begin{figure}
    \centering
    \includegraphics[height=8.5em, trim={5.2cm 2cm 4cm 0cm}, clip]{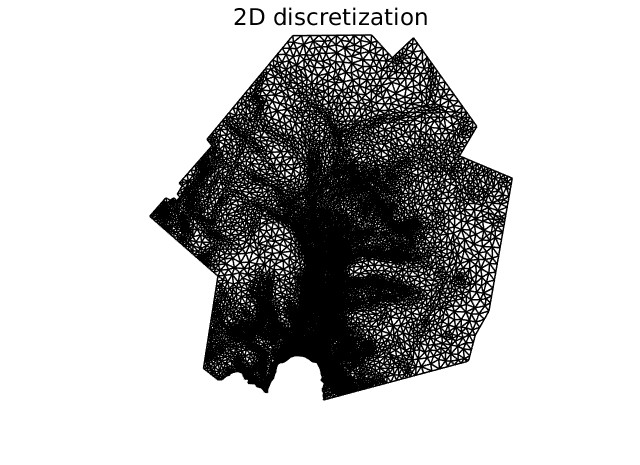}
    \includegraphics[height=8.5em, trim={5cm 2cm 4cm 0cm}, clip]{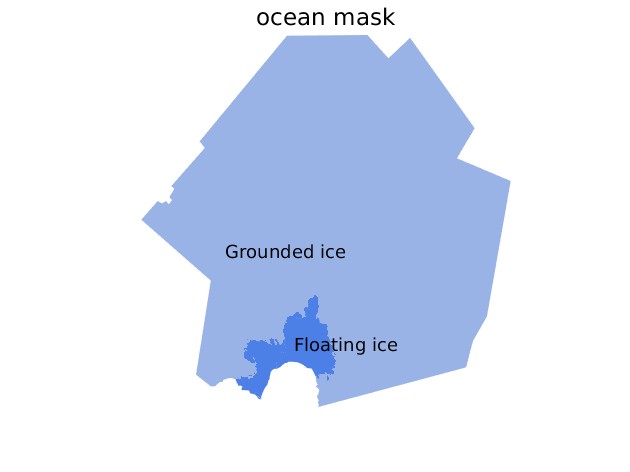}
    \includegraphics[height=8.5em, trim={4cm 2cm 2cm 0cm}, clip]{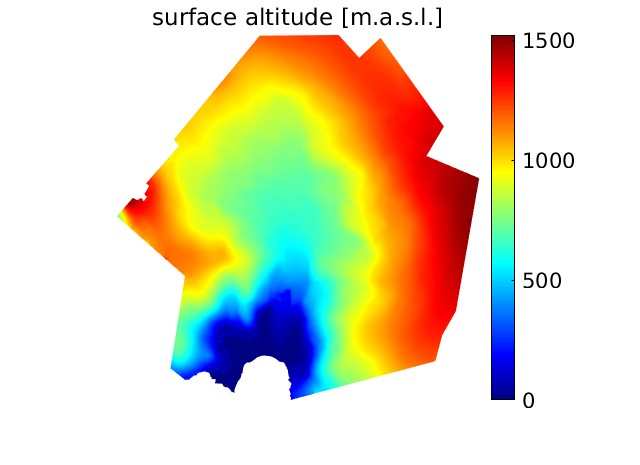}
    \includegraphics[height=8.5em, trim={4cm 2cm 2cm 0cm}, clip]{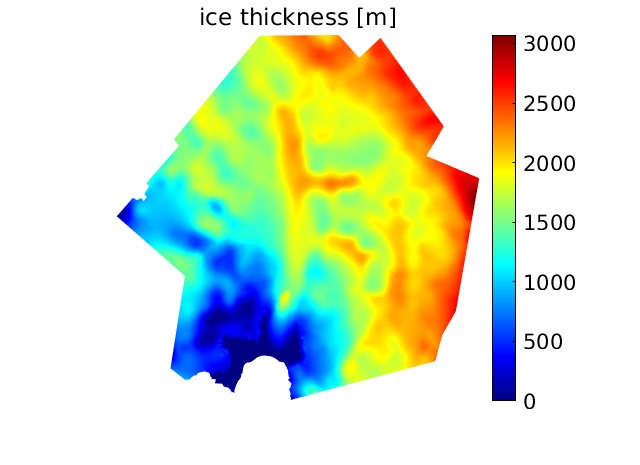}
    \caption{From left to right: Mesh of $D_{2D}$, mask indicating grounded and floating ice, surface altitude in meters above sea level [m.a.s.l.], and ice thickness in meters.}
    \label{fig:PIG:surface}
\end{figure}

Our high-fidelity model is the so-called higher-order (HO) model of the velocity fields $v_x, v_y : D \rightarrow \mathbb{R}$, which characterizes the strain rates
\begin{align*}
    \dot{\varepsilon}_{\rm{HO}, 1}
    &= \left(\begin{array}{c}
    2\frac{\partial v_x}{\partial x} + \frac{\partial v_y}{\partial y} \\
    \frac{1}{2} \frac{\partial v_x}{\partial y} +  \frac{1}{2} \frac{\partial v_y}{\partial x} \\
    \frac{1}{2} \frac{\partial v_x}{\partial z}
    \end{array}\right) &
    \dot{\varepsilon}_{\rm{HO}, 2}
    &= \left(\begin{array}{c}
    \frac{1}{2} \frac{\partial v_x}{\partial y} + \frac{1}{2} \frac{\partial v_y}{\partial x} \\
    \frac{\partial v_x}{\partial x} +  2 \frac{\partial v_y}{\partial y} \\
    \frac{1}{2} \frac{\partial v_y}{\partial z}
    \end{array}\right)
\end{align*}
on $D$ through the partial differential equations
\begin{align}\label{eq:PIG:HO}
    \nabla \cdot (2 \mu \dot{\varepsilon}_{\rm{HO}, 1}) &= \rho_{\rm{ice}} g \frac{\partial s}{\partial x}, &
    \nabla \cdot (2 \mu \dot{\varepsilon}_{\rm{HO}, 2}) &= \rho_{\rm{ice}} g \frac{\partial s}{\partial y},
\end{align}
where $\rho_{\rm{ice}} = 917 \ \textrm{kg}/\textrm{m}^3$ is the ice density, $g = 9.81 \ \textrm{m}/\textrm{s}^2$ is the gravitational acceleration, and $\mu$ is the effective ice viscosity computed from Glen's flow law.
Neglecting the effects of atmospheric pressure, which are small, we apply zero-Neumann boundary conditions on all ice-air interfaces; at the lateral ice-ocean interfaces, we apply water pressure.
At the basal boundary, we prescribe sliding boundary conditions
\begin{align}\label{eq:PIG:HO:boundary}
    2 \mu \dot{\varepsilon}_{\rm{HO}, 1} \cdot \mathbf{n} &= -\beta^2 N v_x &
    2 \mu \dot{\varepsilon}_{\rm{HO}, 2} \cdot \mathbf{n} &= -\beta^2 N v_y
\end{align}
where $\mathbf{n}$ is the outward pointing unit normal, $N$ is the effective pressure, and $\beta : D_{2D} \rightarrow \mathbb{R}_{\ge 0}$ is the basal friction field.
We refer to \cite{aretz2025multifidelity,larour2012continental,Paterson1994} for a complete description of the governing equations; the specific datasets and model parameter laws used in the modeling setup are summarized in Appendix \ref{sec:appendix:PIG}.

For our high-fidelity discretization, we create a triangular mesh of $D_{2D}$ and extrude it vertically  to $D$ with ten evenly-spaced horizontal layers.
The mesh of $D_{2D}$ is shown in Figure~\ref{fig:PIG:surface}, and has a resolution that varies between 250 m and 10 km.
We discretize the HO equations \eqref{eq:PIG:HO} using linear finite elements and denote with $\mathbf{v}_x(\beta), \mathbf{v}_y(\beta) \in \mathbb{R}^{106,350}$ the numerical solution vectors for any given $\beta : D_{2D} \rightarrow \mathbb{R}_{>0}$.
The basal friction field $\beta$ is a major uncertainty in ice sheet simulations because it is not directly observable but inferred from satellite measurements of the ice's surface velocity.
Here, we model $\beta$ to follow a log-normal distribution, i.e., $\theta := \log \beta \sim \mathcal{N}(\log(\beta_0), \mathcal{C}) =: \nu$, with $\beta_0$ inferred from the surface velocity data \cite{rignot2011ice, rignot2011measures} and a double-Laplacian Mat\'ern covariance operator (c.f., Appendix \ref{sec:appendix:PIG}).
We discretize the friction field $\beta$ and the log-friction field $\theta$ with linear finite elements on $D_{2D}$; our discretized parameter dimension is $d=10,635$.

We define our high-fidelity state for a log-friction field $\theta \sim \nu$
\begin{align*}
    \mathbf{u}_0(\theta) := \left(\begin{array}{c}
         \mathbf{v}_x(10^{\theta}) - \mathbf{v}_x(\beta_0) \\
         \mathbf{v}_y(10^{\theta}) - \mathbf{v}_y(\beta_0) 
    \end{array} \right) 
    \in \mathbb{R}^{212,700}
\end{align*}
as the change in ice velocity compared to the velocity $(\mathbf{v}_x(\beta_0), \mathbf{v}_y(\beta_0))^{\top}$ at the parametric mean $\theta_0 := \log(\beta_0)$.
In particular, $\mathbf{u}_0(\theta_0) = \mathbf{0} \in \mathbb{R}^n$ with $n:= 212,700$.
Using ISSM and initializing $\mathbf{v}_x(10^{\theta}), \mathbf{v}_y(10^{\theta})$ at precomputed $\mathbf{v}_x(\beta_0), \mathbf{v}_y(\beta_0)$, the computation of $\mathbf{u}_0(\theta)$ takes on average $14.8$ CPUmin.\footnote{Using eight cores on a dual-socket server equipped with two 64-core AMD EPYC 7H12 processors, totaling 128 physical cores (256 hardware threads) with a base clock speed of 2.60 GHz, and 2.0 TiB of RAM.
}
Figure~\ref{fig:PIG:mean} shows $\theta_0$, the surface layer of velocities $\mathbf{v}_x(\beta_0), \mathbf{v}_y(\beta_0)$, and the changes induced by a random sample $\theta \sim \nu$.

\begin{figure}
    \centering
    \includegraphics[width=\textwidth, trim={11.5cm 4.5cm 7cm 2cm}, clip]{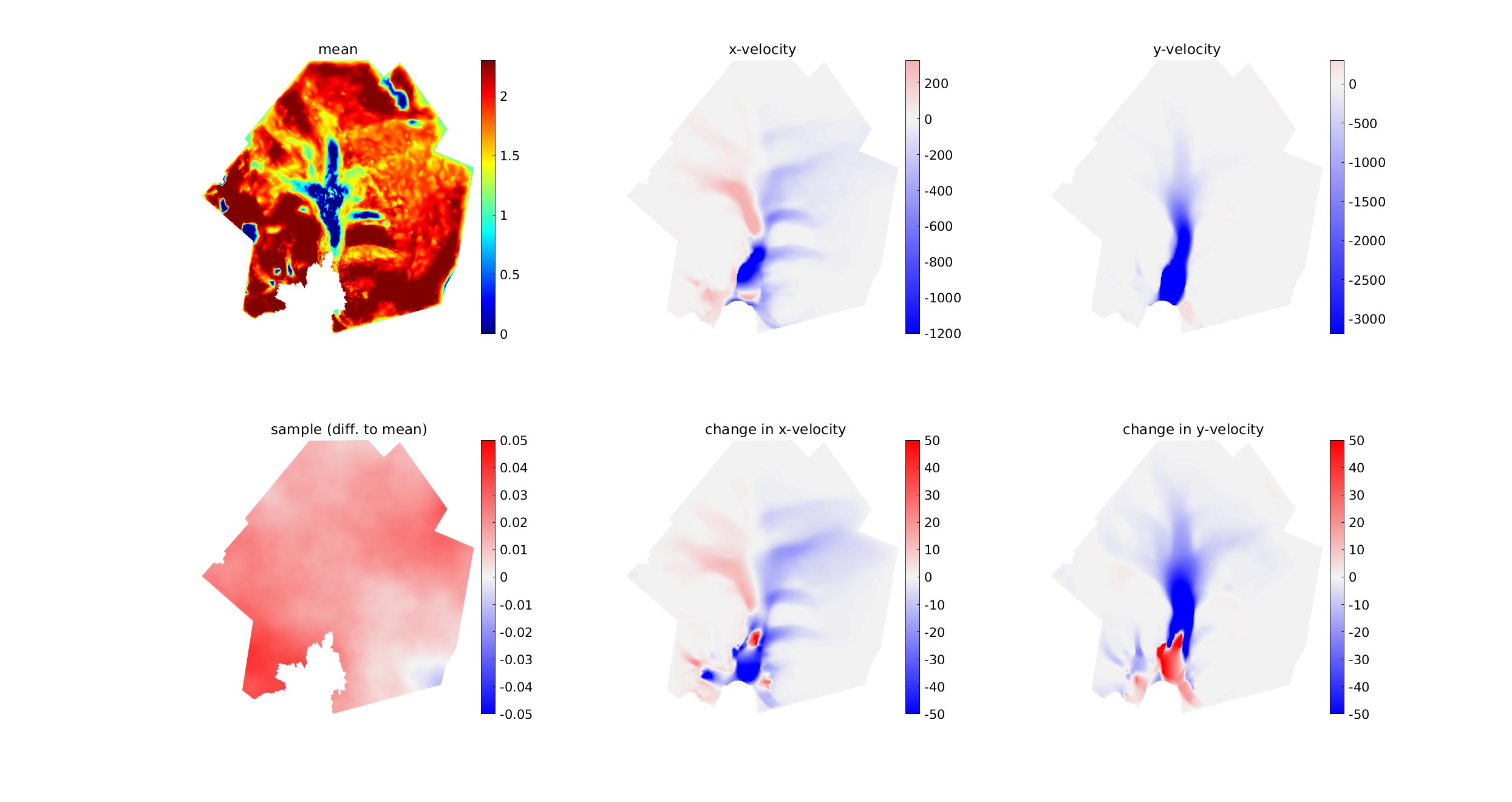}
    \caption{
    Top: Parametric mean of $\theta$ and corresponding ice velocity fields.
    Bottom: Difference between a random sample and parametric mean of $\theta$ (left), and induced changes in velocity fields.}
    \label{fig:PIG:mean}
\end{figure}

Our low-fidelity model employs a simplified physics representation through the Shallow-Shelf Approximation (SSA). SSA models approximations $\overline{v}_x, \overline{v}_y : D_{2D} \rightarrow \mathbb{R}$ to the depth-averaged velocity fields,
\begin{align*}
    \overline{v}_x(x, y) &\approx \frac{1}{h(x, y)} \int_{z=b(x,y)}^{s(x, y)} v_x(x, y, z) dz, &
    \overline{v}_y(x, y) &\approx \frac{1}{h(x, y)} \int_{z=b(x,y)}^{s(x, y)} v_y(x, y, z) dz,
\end{align*}
for all $(x,y)^{\top} \in D_{2D}$.
We refer to \cite{aretz2025multifidelity,larour2012continental,Paterson1994} for the governing equations, but note that while the basal friction field $\beta$ appears in the boundary condition \eqref{eq:PIG:HO:boundary} of the HO model, it is part of a forcing field in the SSA.
Denoting with $\overline{\mathbf{v}}_x(\beta), \overline{\mathbf{v}}_y(\beta) \in \mathbb{R}^{10,635}$ the numerical SSA solution for a given basal friction field, we define the low-fidelity state as
\begin{align}\label{eq:PIG:SSA:u1}
    \mathbf{u}_1(\theta) &:= \bigg(
    \underbrace{
    \overline{\mathbf{d}}_x^{\top}(\theta),
    \dots, 
    \overline{\mathbf{d}}_x^{\top}(\theta)
    }_{10 \times},
    \underbrace{
    \overline{\mathbf{d}}_y^{\top}(\theta),
    \dots, 
    \overline{\mathbf{d}}_y^{\top}(\theta)
    }_{10 \times}
    \bigg)^{\top} \in \mathbb{R}^{212,700}
\end{align}
with $\overline{\mathbf{d}}_x(\theta) := \overline{\mathbf{v}}_x(10^{\theta}) - \overline{\mathbf{v}}_x(\beta_0),
    \overline{\mathbf{d}}_y(\theta) := \overline{\mathbf{v}}_y(10^{\theta}) - \overline{\mathbf{v}}_y(\beta_0) \in \mathbb{R}^{10,635}$.
Here we have used that ISSM orders the finite element coefficient vectors for 3D-fields by layer such that \eqref{eq:PIG:SSA:u1} corresponds to depth-independent velocity fields.
The computation of $\mathbf{u}_1(\theta)$ takes on average $0.144$ CPUmin, using two cores on the same machine as the high-fidelity model. The low-fidelity SSA model thus has an average speed-up of 102$\times$ compared to the high-fidelity HO model.

Because model compute times vary between parameters, we specify the budget $\mybudget$ as the number of permitted high-fidelity solves rather than CPU-time.
This corresponds to the normalization $c_0 := 1$.
For the low-fidelity model, we set $c_1 := 1/72$ as the \textit{worst-case} ratio of high- to low-fidelity CPU-time we observed over a test run.
The resulting ratio $1:72$ of high- to low-fidelity snapshots for any given budget $\mybudget$ is conservative compared to the average $102\times$ speed-up, favoring the high-fidelity POD spaces $\mathbf{V}_0$ in our comparisons below.
For reference, we compute high-fidelity snapshots $\{\mathbf{u}_0(\theta_i^{\rm{ref}})\}_{i=1}^{10,000}$ for 10,000 i.i.d.\ samples $\theta_1^{\rm{ref}}, \dots, \theta_{10,000}^{\rm{ref}} \sim \nu$, and apply the method of snapshots to compute a reference POD basis $\mathbf{V}_{\rm{ref}}$.
Truncating modes after a drop in eigenvalues of more than $10^{-10}$, $\mathbf{V}_{\rm{ref}}$ contains 405 modes. The 50 dominant eigenvalues for this reference POD basis are shown in Figure~\ref{fig:PIG:eigvals:budget5}. Their slow decay highlights the large Kolmogorov n-width of this problem.

\begin{figure}
    \centering
    \includegraphics[width=\linewidth]{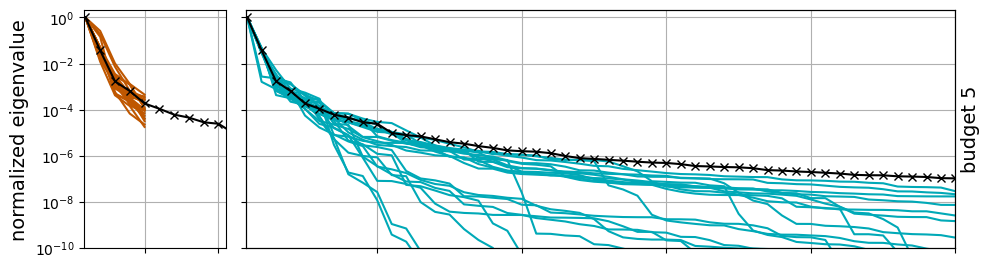}
    \includegraphics[width=\linewidth]{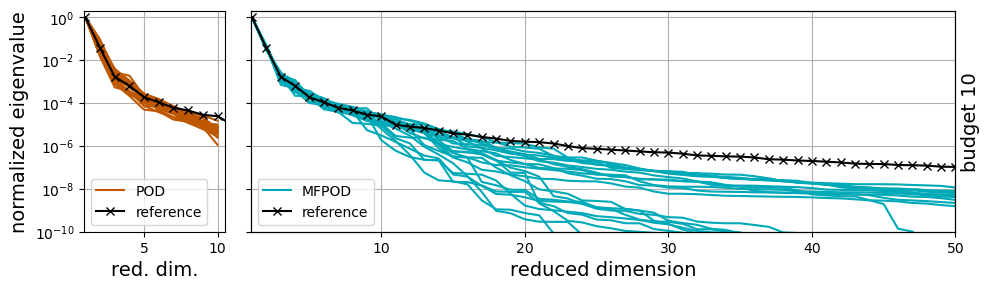}
    \caption{Estimated eigenvalues with high-fidelity POD and MFPOD for 20 independent runs for a budget for $\mybudget=5$ (top) and $\mybudget=10$ (bottom) high-fidelity solves, compared against the reference POD eigenvalues. All spectra are normalized w.r.t.\ their largest eigenvalue.}
    \label{fig:PIG:eigvals:budget5}
\end{figure}

To illustrate the performance of MFPOD, we first choose a computational budget $\mybudget=5$ of five high-fidelity simulations, and compute a high-fidelity POD basis $\mathbf{V}_0$ from five random snapshots of $\mathbf{u}_0$, 
and an MFPOD basis $\mathbf{V}_{\mymf}$ with $m_0 = 2$ high-fidelity and $m_1 = 216$ low-fidelity snapshots.
We repeat the basis computations twenty times with independent draws of the snapshot sets.
The resulting eigenvalues are plotted in the top plot of Figure~\ref{fig:PIG:eigvals:budget5}.
As the figure illustrates, with a budget of $\mybudget=5$, the single-fidelity POD can only identify $r=5$ modes with non-zero eigenvalues. In contrast, MFPOD with the same budget identifies many more modes.  
As shown in the bottom plot of Figure~\ref{fig:PIG:eigvals:budget5}, doubling the budget to $\mybudget=10$ permits POD to identify $r=10$ modes, while MFPOD identifies up to 85 modes (43.7 on average) and tracks well with the size of the reference subspace.

To assess the quality of the identified POD and MFPOD bases, Figure~\ref{fig:PIG:energy} compares the captured energy relative to the reference snapshot set. For comparison, we also include a single-fidelity POD basis computed using only low-fidelity snapshots, with $360$ random snapshots of $\mathbf{u}_1$ for $\mybudget=5$ and $720$ random snapshots of $\mathbf{u}_1$ for $\mybudget=10$. 
For the low-fidelity POD space $\mathbf{V}_1$ (left column), the captured energy plateaus at $99.985\%$ as the reduced dimension increases; increasing the number of snapshots does not change this behavior. 
We note that this plateau falls below the threshold that is typically used in reduced-order modeling to choose the reduced dimension (e.g., the threshold is typically set to $99.99\%$, $99.999\%$,  or more).
The plots clearly show that the POD space $\mathbf{V}_1$ built with the low-fidelity model alone cannot overcome the low-fidelity model's error.
In contrast, the POD space $\mathbf{V}_0$ is restricted through the computational budget, and cannot exceed $r=5$ ($\mybudget=5$) and $r=10$ ($\mybudget=10$).
The (average) captured energy at these dimensions is $99.855\%$ and $99.974\%$, below the plateau for $\mathbf{V}_1$.
The MFPOD method overcomes both restrictions and exceeds the low-fidelity plateau with an average energy captured of $99.995\%$ at $r=20$ for $\mybudget=5$ and $99.987 \%$ at $r=10$ for $\mybudget=10$.
While we do see some outliers of MFPOD for the smaller reduced dimensions, overall the distributions of the captured energy with MFPOD track well with the achievable maximum.

\begin{figure}
    \centering
    \includegraphics[width=\linewidth]{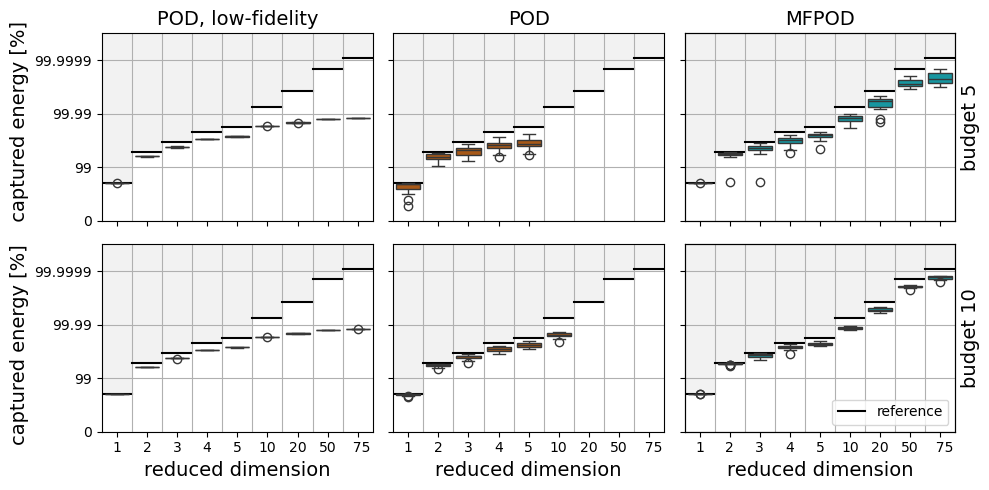}
    \caption{Boxplots of captured energy [\%] by reduced dimension using low-fidelity POD (left), high-fidelity POD (middle), and MFPOD (right) with a computational budget $\mybudget=5$ (top) and $\mybudget=10$ (bottom).
    Boxplots are over 20 independent runs.
    The reference lines mark the captured energy of the reference POD space $\mathbf{V}_{\rm{ref}}$ computed over a reference set of 10,000 snapshots.
    }
    \label{fig:PIG:energy}
\end{figure}

\begin{figure}
    \centering
    \includegraphics[width=\linewidth, trim={8cm 27cm 7cm 4cm},clip]{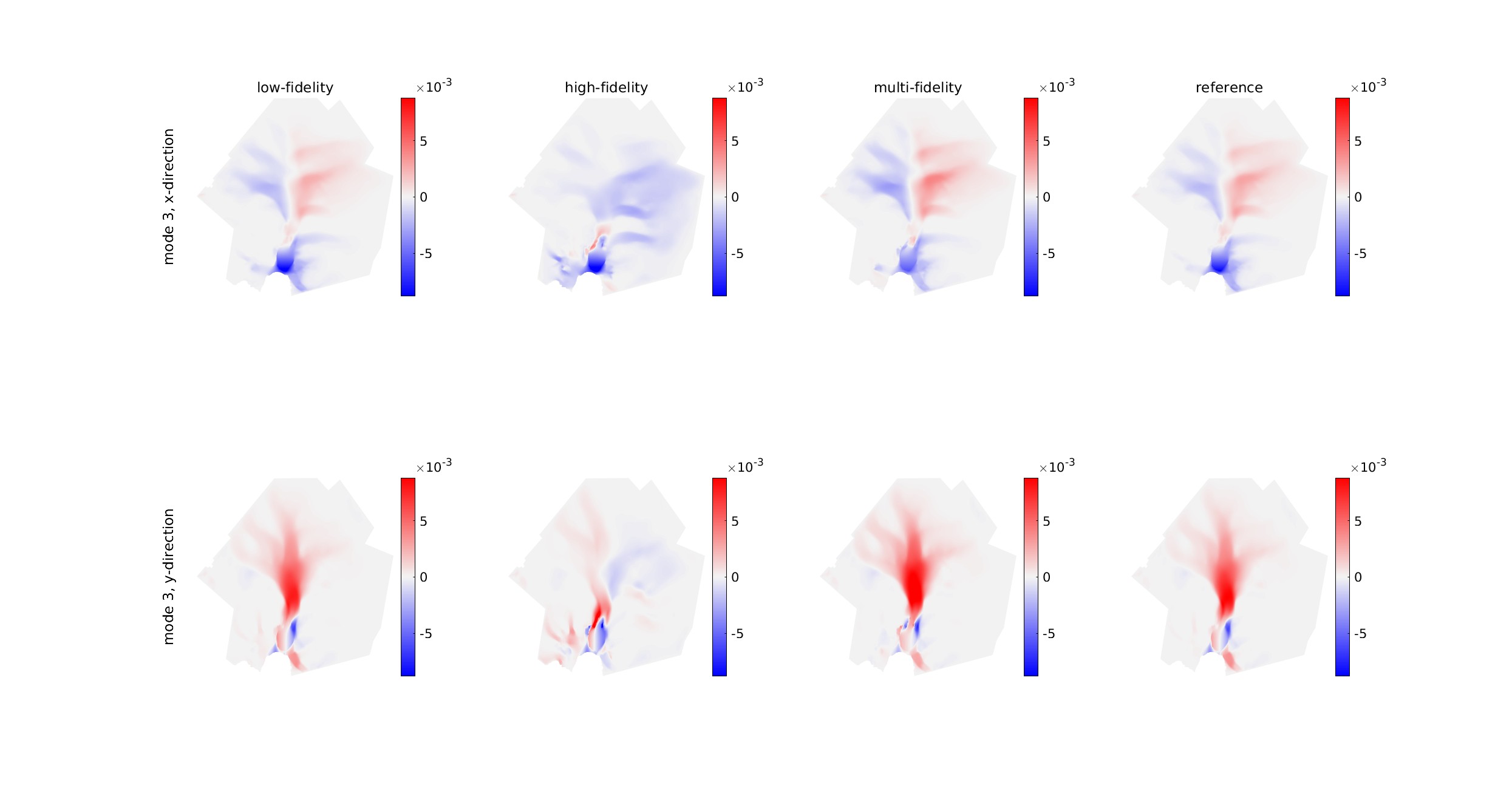}
    \includegraphics[width=\linewidth, trim={8cm 5cm 7cm 26cm},clip]{mode3.jpg}
    \caption{Comparison of the depth-averaged velocity in x- (top) and y-direction (bottom) of the third mode for a budget of five high-fidelity simulations.}
    \label{fig:PIG:mode:3}
\end{figure}

Figure~\ref{fig:PIG:mode:3} contrasts the depth-average of the third mode of $\mathbf{V}_{\rm{ref}}$ to those of $\mathbf{V}_1, \mathbf{V}_0$ and $\mathbf{V}_{\mymf}$ for one of the 20 independent runs.
With POD trained on high-fidelity samples, we already see strong mismatches between the third modes of $\mathbf{V}_1$ and $\mathbf{V}_{\rm{ref}}$, while the third low-fidelity and multifidelity mode both show good agreement.
The low-fidelity mode in particular is visually hardly distinguishable from the reference mode; this is to be expected because the SSA specifically models the depth-averaged velocity field.

\section{Conclusion}
\label{sec:conclusion}

In this work we introduced a multifidelity POD cost function that leverages snapshot information from both a computationally expensive high-fidelity model and cheaper but less accurate low-fidelity models to approximate the expected POD projection error at a lower computational cost.
The modes chosen by minimizing the multifidelity cost function converge in probability to the true POD modes and can be computed by solving an eigenvalue problem.
Our numerical comparisons with single-fidelity POD demonstrate that the MFPOD subspace both overcomes the model error of POD with low-fidelity samples and significantly reduces the generalization error compared to POD with few high-fidelity samples.
For a numerical model of the ice velocity of the Pine Island Glacier in Western Antarctica, MFPOD performed similarly to POD at a $10\times$ reduced offline sampling cost.
Future work will investigate optimal budget distribution strategies based on the eigenvalue decay of pilot samples from each model to further decrease the variance of the MFPOD estimator.

\appendix
\section{Proofs}\label{sec:appendix}

In this section we provide the proofs for Lemma \ref{thm:costfunction}, Theorem \ref{thm:convergence:C}, and Proposition \ref{thm:convergence:lambda}.
We start with Lemma \ref{thm:costfunction}, which connects the multifidelity cost function $\mathcal{J}_{\mymf}$ from \eqref{eq:Jmf:indefinite} to the eigenvalue-eigenfunction pairs of $\mathcal{C}_{\mymf}$ from \eqref{eq:Cmf}.

\begin{lemma}
    Define $\mathcal{C}_{\mymf} : H \rightarrow H$ through \eqref{eq:Cmf} and let $(\lambda_j, v_j) \in \mathbb{R} \times H$, $1 \le j \le \dim H \in \mathbb{N} \cup \{\infty\}$, be its eigenvalue-eigenfunction pairs, ordered such that $\lambda_1 \ge \dots \ge \lambda_r \ge \sup_{j > r} \lambda_j$. 
    Then $\mathcal{J}_{\mymf}(V) = \sum_{j=1}^{\dim H} \lambda_j (1-\| \Pi_V v_j\|_H^2)$.
In particular, $\text{span}(v_1, \dots, v_r) \in \argmin_{V \subset H, \dim V = r} \mathcal{J}_{\mymf}(V)$,
and the minimizer is unique iff $\lambda_r > \sup_{j > r} \lambda_j$.
\end{lemma}

\begin{proof}
Because $H$ is separable, and because $\mathcal{C}_{\mymf}$ is self-adjoint with finite-di\-men\-sion\-al range (and is hence compact), the eigenfunctions $\{v_j\}_{j=1}^{\dim H}$ form an orthonormal basis of $H$.
Thus, for $0 \le \ell \le L$ and $1 \le i \le m_{\ell}$,
\begin{align*}
    \|u_{\ell}(\theta_i)-\Pi_V u_{\ell}(\theta_i)\|_H^2
    = \sum_{j=1}^{\dim H} (\Pi_{V_{\perp}} u_{\ell}(\theta_i), v_j)_H^2
    = \sum_{j=1}^{\dim H} ( u_{\ell}(\theta_i), \Pi_{V_{\perp}} v_j)_H^2
\end{align*}
where $V_{\perp}$ is the orthogonal complement of $V$ in $H$.
Then
\begin{align*}
    &\mathcal{J}_{\mymf}(V)
    = \sum_{j=1}^{\dim H}(\mathcal{C}_{\mymf} \Pi_{V_{\perp}} v_j, \Pi_{V_{\perp}} v_j)_H 
    = \sum_{j=1}^{\dim H} \sum_{k_1=1}^{\dim H}(\Pi_{V_{\perp}} v_j, v_{k_1})_H(\mathcal{C}_{\mymf} v_{k_1}, \Pi_{V_{\perp}} v_j) \\
    &\quad= \sum_{j=1}^{\dim H} \sum_{k_1, k_2=1}^{\dim H} (\Pi_{V_{\perp}} v_j, v_{k_1})_H (\Pi_{V_{\perp}} v_j, v_{k_2})_H (\mathcal{C}_{\mymf}v_{k_1}, v_{k_2})_H \\
    &\quad= \sum_{j=1}^{\dim H} \sum_{k=1}^{\dim H} \lambda_{k} (\Pi_{V_{\perp}} v_j, v_k)_H^2 
    = \sum_{k=1}^{\dim H} \lambda_k \|\Pi_{V_{\perp}} v_k\|_H^2 = \sum_{j=1}^{\dim H} \lambda_j (1-\| \Pi_V v_j\|_H^2).
\end{align*}
\end{proof}

We next show the proof of Theorem \ref{thm:convergence:C}, which states that there exists a constant $\gamma < \infty$ such that $\mathbb{E}[\|\mathcal{C}_{\mymf}-\mathcal{C}\|_{\rm{HS}}^2] = \gamma / m_0$ for all $m_0 \in \mathbb{N}$.
Note that the following proof gives an explicit expression for $\gamma$.

\begin{proof}[Proof of Theorem \ref{thm:convergence:C}]
Let $m_0 \in \mathbb{N}$ be arbitrary, and abbreviate $\mathcal{C}_{\mymf} := \mathcal{C}_{\mymf}(m_0)$.
For $\ell = 0, \dots, L$ and $j,k=1, \dots, \dim H$, define $f_{\ell}^{(j,k)}(\theta):= (u_{\ell}(\theta), v_j^*)_H(u_{\ell}(\theta), v_k^*)_H$ for brevity of notation.
Noting that for arbitrary integers $j, k$ with $1 \le j, k \le \dim H$
\begin{align*}
    (\mathcal{C}_{\mymf}v_j^*, v_k^*)_H
    &= \frac{1}{m_0} \sum_{i=1}^{m_0} f_{0}^{(j,k)}(\theta_i)
    + \sum_{\ell=1}^L \bigg(
    \frac{\alpha_{\ell}}{m_{\ell}} \sum_{i=1}^{m_{\ell}} f_{\ell}^{(j,k)}(\theta_i)
    - \frac{\alpha_{\ell}}{m_{\ell-1}} \sum_{i=1}^{m_{\ell-1}} f_{\ell}^{(j,k)}(\theta_i)
    \bigg)
\end{align*}
is a multifidelity estimator with expectation 
\begin{align*}
    \mathbb{E}[(\mathcal{C}_{\mymf}v_j^*, v_k^*)_H]
    = \mathbb{E}_{\theta}[f_{0}^{(j,k)}(\theta)]
    = \mathbb{E}_{\theta}[(u_0(\theta), v_j^*)_H(u_0(\theta), v_k^*)_H]
    = (\mathcal{C}v_j^*, v_k^*)_H,
\end{align*}
we obtain from \cite{peherstorfer2016optimal} (Lemma 3.3) that
\begin{align*}
    &\mathbb{E}[((\mathcal{C}_{\mymf}-\mathcal{C})v_j^*, v_k^*)_H^2]
    = \mathbb{V}[(\mathcal{C}_{\mymf}v_j^*, v_k^*)_H] \\
    &= \frac{1}{m_0}
    \mathbb{V}[f_{0}^{(j,k)}] + \sum_{\ell = 1}^L \bigg(
    \frac{1}{m_0q_{\ell-1}} - \frac{1}{m_0q_{\ell}}
    \bigg)\bigg(
    \alpha_{\ell}^2\mathbb{V}[f_{\ell}^{(j,k)}] - 2\alpha_{\ell} \text{Cov}(f_{0}^{(j,k)}, f_{\ell}^{(j,k)})
    \bigg) \\
    &= \frac{1}{m_0}\bigg(\mathbb{V}[f_{0}^{(j,k)}] 
    + \sum_{\ell = 1}^L \bigg(
    \frac{1}{q_{\ell-1}} - \frac{1}{q_{\ell}}
    \bigg)\bigg(
    \mathbb{V}\big[f_{0}^{(j,k)}-\alpha_{\ell}f_{\ell}^{(j,k)}\big] - \mathbb{V}\big[f_{0}^{(j,k)}\big]
    \bigg)
    \bigg) \\
    &= \frac{1}{m_0}\bigg(
    \frac{1}{q_L}\mathbb{V}[f_{0}^{(j,k)}] 
    + \sum_{\ell = 1}^L \bigg(
    \frac{1}{q_{\ell-1}} - \frac{1}{q_{\ell}}
    \bigg) \mathbb{V}\big[f_{0}^{(j,k)}-\alpha_{\ell}f_{\ell}^{(j,k)}\big] \bigg).
\end{align*}
We define the $m_0$-independent constant
\begin{align}\label{eq:gamma}
    \gamma := \sum_{j,k=1}^{\dim H} \bigg(
    \frac{1}{q_L}\mathbb{V}[f_{0}^{(j,k)}] 
    + \sum_{\ell = 1}^L \bigg(
    \frac{1}{q_{\ell-1}} - \frac{1}{q_{\ell}}
    \bigg) \mathbb{V}\big[f_{0}^{(j,k)}-\alpha_{\ell}f_{\ell}^{(j,k)}\big] \bigg) \ge 0.
\end{align}
Note that $\gamma < \infty$ because, for each $\ell = 0, \dots, L$, 
\begin{align*}
    \sum_{j,k=1}^{\dim H} \mathbb{V}[f_{0}^{(j,k)}]
    &\le \sum_{j,k=1}^{\dim H} \mathbb{E}[(f_{0}^{(j,k)})^2] 
    = \mathbb{E}\bigg[\sum_{j,k=1}^{\dim H}  (u_{\ell}(\theta), v_j^*)_H^2 (u_{\ell}(\theta), v_k^*)_H^2\bigg] \\
    &= \mathbb{E}\bigg[\bigg(
    \sum_{j=1}^{\dim H}  (u_{\ell}(\theta), v_j^*)_H^2
    \bigg)\bigg(
    \sum_{k=1}^{\dim H}  (u_{\ell}(\theta), v_k^*)_H^2
    \bigg)\bigg] 
    = \mathbb{E}\big[\|u_{\ell}\|_H^4\big] < \infty,
\end{align*}
and $\mathbb{V}\big[f_{0}^{(j,k)}-\alpha_{\ell}f_{\ell}^{(j,k)}\big] \le 2\mathbb{V}\big[f_{0}^{(j,k)}\big] + 2 \alpha_{\ell}^2\mathbb{V}\big[f_{\ell}^{(j,k)}\big]$.

To show the claim \eqref{eq:convergence:C}, we first note that $\mathcal{C}$ is trace-class because $\mathbb{E}[\|u_0(\theta)\|_H^2] < \infty$, and that, for any draw of samples the range of $\mathcal{C}_{\mymf}$ is finite-dimensional.
Thus $\|\mathcal{C}_{\mymf}-\mathcal{C}\|_{\rm{HS}} < \infty$.
Taking the expectation yields
\begin{align*}
    \mathbb{E}\big[\|\mathcal{C}_{\mymf}-\mathcal{C}\|_{\rm{HS}}^2\big]
    &= \mathbb{E}\bigg[\sum_{j,k=1}^{\dim H} ((\mathcal{C}_{\mymf}-\mathcal{C})v_j^*, v_k^*)_H^2\bigg]
    = \sum_{j,k=1}^{\dim H} \mathbb{V}[(\mathcal{C}_{\mymf}v_j^*, v_k^*)_H] 
    = \frac{1}{m_0} \gamma .
\end{align*}
The result follows because $m_0$ was arbitrary and $\gamma < \infty$ is independent of $m_0$.
\end{proof}

Building upon Theorem \ref{thm:convergence:C}, we prove Proposition \ref{thm:convergence:lambda}, providing upper bounds for the MSEs $\mathbb{E}[ (\sum_{j=1}^r \lambda_j -\sum_{j=1}^r \lambda_j^*)^2 ]$ and $\mathbb{E}[(\sum_{j=1}^r \|v_j - \Pi_{V_*}v_j\|_H^2)^2]$.

\begin{proof}[Proof of Proposition \ref{thm:convergence:lambda}]
Let $m_0 \in \mathbb{N}$ be arbitrary, and define the linear, self-adjoint operator $\mathcal{E} := \mathcal{C} - \mathcal{C}_{\mymf}(m_0) : H \rightarrow H$.
Let $\delta_j \in \mathbb{R}$, $1 \le j \le \dim H$, be its eigenvalues ordered such that $|\delta_1| \ge |\delta_2| \ge \dots \ge 0$.
If $\sum_{j=1}^r \lambda_j^* < \sum_{j=1}^r \lambda_j$,
then
\begin{align*}
    \bigg|
    \sum_{j=1}^r \lambda_j^* - \sum_{j=1}^r \lambda_j
    \bigg|
    &= \sum_{j=1}^r (\mathcal{C}_{\mymf}v_j, v_j)_H - \sum_{j=1}^r (\mathcal{C}v_j^*, v_j^*)_H \\
    &= \sum_{j=1}^r (\mathcal{C}_{\mymf}v_j, v_j)_H - \sup_{\substack{w_1, \dots, w_r\\\text{orthonormal}}} \sum_{j=1}^r (\mathcal{C}w_j, w_j)_H \\
    &\le -\sum_{j=1}^r (\mathcal{E}v_j, v_j)_H 
    \le \sum_{j=1}^r |(\mathcal{E}v_j, v_j)_H|
    \le \sum_{j=1}^r |\delta_j|,
\end{align*}
where we have used that $v_1, \dots, v_r$ are orthonormal.
For $\sum_{j=1}^r \lambda_j^* > \sum_{j=1}^r \lambda_j$,
then $|\sum_{j=1}^r \lambda_j^* - \sum_{j=1}^r \lambda_j| \le \sum_{j=1}^r |\delta_j|$ follows analogously.
Using that, by the Cauchy-Schwarz inequality, $(\sum_{j=1}^r |\delta_j|)^2 \le r \sum_{j=1}^r |\delta_j|^2 \le r \sum_{j=1}^{\dim H} |\delta_j|^2 = r \|\mathcal{\mathcal{E}}\|_{\rm{HS}}^2$,
\eqref{eq:bound:eigvals} follows from Theorem \ref{thm:convergence:C}.

To prove \eqref{eq:bound:eigvecs}, we first use the triangle inequality to bound
\begin{equation}\label{eq:bound:eigvecs:proof:upper}
    \begin{aligned}
    \bigg|\sum_{j=1}^r (\mathcal{C}v_j, v_j)_H - \sum_{j=1}^r\lambda_j^*\bigg| 
    &\le \bigg|\sum_{j=1}^r (\lambda_j - \lambda_j^*)\bigg| + \sum_{j=1}^r 
    |(\mathcal{E} v_j, v_j)_H| 
    \le 2 \sum_{k=1}^r |\delta_k|.
\end{aligned}
\end{equation}
In the reverse direction, we use that $\{v_k^*\}_{k=1}^{\dim H}$ is an orthonormal basis of $H$ and that $\|v_1\|_H = \dots = \|v_r\|_H = 1$ for
\begin{align*}
&\bigg|\sum_{j=1}^r (\mathcal{C}v_j, v_j)_H -\sum_{j=1}^r
     \lambda_j^*\bigg| 
= \bigg|\sum_{j=1}^r 
     \lambda_j^* - \sum_{j=1}^r \sum_{k=1}^{\dim H} (v_j, v_k^*)_H (\mathcal{C}v_k^*, v_j)_H \bigg| \\
&\qquad \ge \sum_{j=1}^r 
     \lambda_j^*- \sum_{j=1}^r\sum_{k=1}^{r} \lambda_k^*(v_j, v_k^*)_H^2 - \sum_{j=1}^r\sum_{k=r+1}^{\dim H} \lambda_k^*(v_j, v_k^*)_H^2 \\
&\qquad \ge \sum_{k=1}^r 
     \lambda_k^* \big(1-\sum_{j=1}^r (v_j, v_k^*)_H^2\big)
     - \lambda_{r+1}^*\sum_{j=1}^r\sum_{k=r+1}^{\dim H} (v_j, v_k^*)_H^2 \\
&\qquad \ge \lambda_r^* \sum_{k=1}^r \big(1-\sum_{j=1}^r (v_j, v_k^*)_H^2\big)
     - \lambda_{r+1}^*\sum_{j=1}^r\big(1-\sum_{k=1}^{r} (v_j, v_k^*)_H^2\big) \\
&\qquad = (\lambda_r^*-\lambda_{r+1}^*)\sum_{j=1}^r \bigg(1 -  \sum_{k=1}^{r} (v_j, v_k^*)_H^2\bigg) 
= (\lambda_r^*-\lambda_{r+1}^*)\sum_{j=1}^r \|v_j - \Pi_{V_*}v_j\|_H^2.
\end{align*}
The claim \eqref{eq:bound:eigvecs} then follows with \eqref{eq:bound:eigvecs:proof:upper} and Theorem \ref{thm:convergence:C}.
\end{proof}

\section{Pine Island glacier datasets}\label{sec:appendix:PIG}

We provide additional information on the datasets used for the Pine Island glacier example in Section~\ref{sec:results:pig}.
The surface altitude and ice thickness fields $s$ and $h$ were obtained from a ten year spin-up run initialized at the satellite surface altitude dataset \cite{bamber2009new, griggs2009newPart2} with the average surface-mass-balance field from \cite{vaughan1999reassessment} and basal topography from \cite{nitsche2007bathymetry}.
The effective ice viscosity $\mu$ is modeled with Glen's flow law $\mu = \frac{B}{\sqrt{2}}\| (\dot{\varepsilon}_{\rm{HO}, 1}, \dot{\varepsilon}_{\rm{HO}, 2}) \|^{-2/3}$ with ice rigidity $B$ computed from the temperature field in \cite{comiso2000variability}.
The effective pressure $N$ was computed according to Budd's friction law \cite{Budd1979}.
The local resolution of the mesh was chosen based on the surface velocity data from \cite{rignot2011ice, rignot2011measures} and is similar to the resolution used in the sensitivity study \cite{seroussi2014sensitivity}.
The log-friction mean $\beta_0$ was inferred from the surface velocity data \cite{rignot2011ice, rignot2011measures}.
The covariance operator is a double-Laplacian Mat\'ern covariance operator (see \cite{VillaPetraGhattas21} with parameters $\gamma=1, \delta=8$).

\section*{Acknowledgments}
We would like to thank  Anirban Chaudhuri, Tiangang Cui, Judy Hao, Graham Pash, and other members of the Willcox research group for fruitful discussions.
This work was supported in parts by the Department of Energy grant DE-SC002317, Air Force Office of Scientific Research MURI grant FA9550-24-1-0327, and the National Science Foundation grant \#2103942.

\bibliographystyle{plain}
\bibliography{references}
\end{document}